\pdfoutput=1
\RequirePackage{ifpdf}
\ifpdf 
\documentclass[pdftex]{sigma}
\else
\documentclass{sigma}
\fi

\usepackage{enumitem}
\usepackage{mathrsfs}
\usepackage{stmaryrd}
\usepackage{tikz}

\usepackage{multicol}
\usepackage{tikz-cd}
\usetikzlibrary{arrows}

\makeatletter
\DeclareFontFamily{U}{tipa}{}
\DeclareFontShape{U}{tipa}{m}{n}{<->tipa10}{}
\newcommand{\arc@char}{{\usefont{U}{tipa}{m}{n}\symbol{62}}}%

\newcommand{\arc}[1]{\mathpalette\arc@arc{#1}}

\newcommand{\arc@arc}[2]{%
 \sbox0{$\m@th#1#2$}%
 \vbox{
 \hbox{\resizebox{\wd0}{\height}{\arc@char}}
 \nointerlineskip
 \box0
 }%
}
\makeatother

\usepackage{scalerel,stackengine}
\stackMath
\newcommand\reallywidehat[1]{%
\savestack{\tmpbox}{\stretchto{%
 \scaleto{%
 \scalerel*[\widthof{\ensuremath{#1}}]{\kern-.6pt\bigwedge\kern-.6pt}%
 {\rule[-\textheight/2]{1ex}{\textheight}}
 }{\textheight}%
}{0.5ex}}%
\stackon[1pt]{#1}{\tmpbox}%
}

\DeclareSymbolFont{bbold}{U}{bbold}{m}{n}
\DeclareSymbolFontAlphabet{\mathbbold}{bbold}

\usepackage{bbm}

\numberwithin{equation}{section}

\newtheorem{Theorem}{Theorem}[section]
\newtheorem*{Theorem*}{Theorem}

\newtheorem{Lemma}[Theorem]{Lemma}
\newtheorem{Proposition}[Theorem]{Proposition}

\theoremstyle{definition}
\newtheorem{Definition}[Theorem]{Definition}

\newtheorem{Example}[Theorem]{Example}
\newtheorem{Remark}[Theorem]{Remark}
\newtheorem{Notation}[Theorem]{Notation}

\newcommand{\mbb}[1]{\mathbb{#1}}

\newcommand{\mfk}[1]{\mathfrak{#1}}

\newcommand{\Ext}{\textrm{Ext}}

\newcommand{\N}{\mbb{N}}
\newcommand{\Z}{\mbb{Z}}

\newcommand{\C}{\mbb{C}}

\newcommand{\vac}{\mathbbold{1}}
\newcommand{\pd}{\partial}
\newcommand{\eval}[2][v]{\lrgl{#1',#2#1}}
\newcommand{\db}{\mathbf{d}}
\newcommand{\Db}{\mathcal{D}}

\newcommand{\lrgl}[1]{\langle#1\rangle}

\newcommand{\im}{\textrm{im}}

\newcommand{\sseq}{\subseteq}

\newcommand{\iv}{^{-1}}

\newcommand{\zsum}[1]{\sum_{#1\in\Z}}

\newcommand{\g}{{\mathfrak g}}

\newcommand{\Hom}{\textrm{Hom}}
\newcommand{\End}{\textrm{End}}

\newcommand{\Der}{\textrm{Der}}

\newcommand{\one}{\mathbf{1}}
\renewcommand{\d}{\mathbf{d}}

\newcommand{\slt}{\mathfrak{sl}_2}

\newcommand{\Res}{\text{Res}}

\newcommand{\Sym}{\text{Sym}}

\setenumerate{label=\arabic*)}

\begin{document}

\allowdisplaybreaks

\newcommand{\arXivNumber}{2510.17293}

\renewcommand{\thefootnote}{}

\renewcommand{\PaperNumber}{103}

\FirstPageHeading

\ShortArticleName{A Cohomology Theory of Supercommutative Algebras}

\ArticleName{A Cohomology Theory of Supercommutative Algebras\\ and Grading-Restricted Vertex Superalgebras\footnote{This paper is a~contribution to the Special Issue on Recent Advances in Vertex Operator Algebras in honor of James Lepowsky. The~full collection is available at \href{https://sigma-journal.com/Lepowsky.html}{https://sigma-journal.com/Lepowsky.html}}}

\Author{Paul JOHNSON~$^{\rm a}$ and Fei QI~$^{\rm b}$}

\AuthorNameForHeading{P.~Johnson and F.~Qi}

\Address{$^{\rm a)}$~Department of Mathematics, University of Denver, Denver, CO 80210, USA}
\EmailD{\mail{paul.t.johnson@du.edu}, \mail{paul.j.algebra@outlook.com}}

\Address{$^{\rm b)}$~School of Mathematics (Zhuhai), Sun Yat-Sen University, Zhuhai, Guangdong, P.R.~China}
\EmailD{\mail{qifei@mail.sysu.edu.cn}, \mail{fei.qi.math.phys@gmail.com}}

\ArticleDates{Received October 25, 2025, in final form October 15, 2026; Published online October 29, 2026}

\Abstract{This paper constructs the cohomology theory for grading-restricted vertex superalgebras, generalizing Yi-Zhi Huang's cohomology theory of grading-restricted vertex algebras. To simplify the discussion, motivate the construction, and make it easier for the reader to understand the technical points, we also include the construction of the cohomology theory of supercommutative associative algebras, a generalization of the Harrison cohomology theory of a commutative algebra that has not been explicitly written down. The paper will serve as the foundation for many subsequent studies, especially the deformation theory of vertex superalgebras.}


\Classification{17B69; 16E40; 13D03}

\renewcommand{\thefootnote}{\arabic{footnote}}
\setcounter{footnote}{0}

\begin{flushright}
\begin{minipage}{60mm}
\it Dedicated to James Lepowsky \\
on the occasion of his 80th birthday
\end{minipage}
\end{flushright}

\section{Introduction}
A vertex superalgebra is an algebraic structure formed by vertex operators satisfying supercommutativity and associativity. It plays a fundamental role in mathematical physics, especially in the mathematical construction of two-dimensional superconformal field theories. So far, most existing results require the module category of the vertex superalgebra to be semisimple. The results for vertex superalgebras with nonsemisimple module categories are limited and often rely heavily on particular properties of the corresponding examples.

In the study of nonsemisimple associative algebras, the cohomology method has played a~crucial role. In representation theory, the first Hochschild cohomology associated with a certain bimodule characterizes the nontrivial extensions of two modules. In deformation theory, the second Hochschild (and Harrison) cohomology describes square-zero deformations and first-order deformations in the category of associative (and commutative associative, respectively) algebras, while the third Hochschild (and Harrison) cohomology describes obstructions to integrating a~first-order deformation. Since vertex operators in vertex superalgebras satisfy associativity, we naturally wish to develop a cohomology theory for vertex superalgebras that is analogous to that for an associative algebra, and establish a general method for studying representation theory and deformation theory of vertex superalgebras.

A lot of work has been done on vertex algebras, which are $\mathbb{Z}$-graded with vertex operators satisfying commutativity. In \cite{H-Coh}, Yi-Zhi Huang constructed a cohomology theory for a~grading-restricted vertex algebra $V$ and its module $W$, where the vertex operators satisfy both commutativity and associativity. The key idea is to use the space $\widetilde{W}_{z_1,\dots, z_n}$ of $\overline{W}$-valued rational functions, the analytic continuations of series with coefficients in $W$. The cochain complex is defined by linear maps from $V^{\otimes n} \to \widetilde{W}_{z_1, \dots, z_n}$ that are compatible with the grading and translation operator, and satisfy very technical but natural convergence conditions (called the composable condition). Even though different summands in the coboundary operator have disjoint regions of convergence, after analytic continuation, we operate in the space of $\overline{W}$-valued rational functions, and thus make sense of the coboundary operator.

Based on this construction, in \cite{H-1st-2nd-Coh}, Huang proved that the first cohomology describes derivations of the vertex algebra, and the second cohomology describes the square-zero extensions and first-order deformations. In \cite{Q-Coh}, the second author generalized Huang's construction to meromorphic open-string vertex algebras, where the vertex operators satisfy only associativity without any commutativity. In \cite{HQ-Red}, Yi-Zhi Huang and the second author gave a cohomological criterion of semisimplicity. In \cite{Q-Ext-1}, the second author gave a cohomological description of the $\Ext^1$-group of two left-modules for a meromorphic open-string vertex algebra. In \cite{KQ}, Vladimir Kovalchuk and the second author established an algorithm for classifying first-order deformations for all freely generated vertex algebras.

All these results are expected to hold for vertex superalgebras. Moreover, in recent discussions and computations with Andrew Linshaw and Yi-Zhi Huang, the second author believes that we are capable of classifying full formal deformations of a large class of strongly generated vertex superalgebras. Of particular interest are the $W$-algebras, which are generally $\frac{1}{2}\Z$-graded and might have both bosonic and fermionic generating fields of half-integral weights. To accommodate such fields, it is necessary to extend Huang's construction of cohomology theory to the most general vertex superalgebras. To solidify the foundation for all the subsequent generalizations, it is essential to examine all the details carefully. This paper serves the purpose.

Passing from commutativity to supercommutativity is not as easy as it seems. For the Harrison cohomology of a commutative associative algebra, the corresponding cochain complex is obtained from imposing a shuffle sum condition on the space of Hochschild cochains. For a supercommutative associative algebra, we need to modify each summand in the shuffle sum condition by an additional sign factor to keep track of the positions of the odd elements. This additional sign factor proves to be quite confusing, as some of the most obvious expectations actually fail (see Remarks \ref{Rema-sigma-o}\,(2)). For this reason, we give a self-contained treatment of these signs in Section~\ref{section2}, in a form suited to the subsequent construction for vertex superalgebras.\footnote{It was revealed by ChatGPT that related treatments existed in~\cite{K,Gohr}. But we did not conduct a detailed comparison.}

For the convenience of the reader, we have decided to organize the paper as follows: In Section~\ref{section2}, we elaborate on the details of the generalization of Harrison cohomology to a supercommutative associative algebra. In Section~\ref{section3}, we construct the cohomology theory for grading-restricted vertex superalgebras. In Section~\ref{section4}, we study the first and second cohomologies. In Section~\ref{section5}, we provide an explicit calculation of the first cohomology of the vertex superalgebra~$V^k(\mfk{osp}(1|2))$.

\section{Cohomology theory for supercommutative algebras}\label{section2}

In this section, we give the definitions of superalgebras and supercommutative algebras, which are particular types of associative algebras. To define the cohomology theory, we introduce a~shuffle condition on the Hochschild cochain complex. Cochains satisfying the shuffle condition form a Harrison-like cochain complex for a supercommutative algebra.
All vector spaces will be assumed to be over $\C$, although the results of this section are easy to generalize to other fields. We use $\N$ to denote the set of nonnegative integers $\{0,1,2,\dots\}$ and $\Z_+$ to denote the set of positive integers $\{1,2,\dots\}$.

\subsection{Superalgebra and supermodules}

\begin{Definition}
 A \textit{superalgebra} $A$ is a $\Z_2$-graded vector space (a vector superspace) $A=A^0\oplus A^1$ (hereafter the $\oplus$ means the direct sum of vector spaces) equipped with a bilinear multiplication $\cdot\colon A\otimes A\to A$ such that
$A^i\cdot A^j\sseq A^{i+j}$.
 Elements of $A^0$ are called \textit{even} and elements of~$A^1$ \textit{odd}. Given a homogeneous element $a\in A^i$, we say $|a|=i$, and call $|a|$ the \textit{parity} of~$a$. A~\textit{supercommutative associative superalgebra} is a superalgebra whose multiplication is associative and satisfies supercommutativity: for every homogeneous $a, b\in A$,
$b\cdot a=(-1)^{|a|\cdot|b|}a\cdot b$.
 For convenience, we shall use the abbreviation \textit{supercommutative algebra}.
\end{Definition}

\begin{Remark}
 Note that in a supercommutative algebra $A$, an even element commutes with every element in $A$. Only for two odd elements $a, b\in A$, we have the anticommutativity $ba = -ab$.
\end{Remark}

\begin{Definition}
 We say $M$ is a \textit{supermodule} for the supercommutative algebra $A$ if $M$ is a~$\Z_2$\nobreakdash-graded vector space $M = M^0 \oplus M^1$ equipped with an action $\cdot\colon A \otimes M \to M$, such that $M$ forms a module for the associative algebra $A$, and
$A^i \cdot M^j \subseteq M^{i+j}$.
 We similarly define even elements, odd elements, and the parity $|m|=i$ for homogeneous elements $m \in M^i$.
\end{Definition}

\begin{Remark}\label{rema-right-action}
A supermodule $M$ may be regarded as a bimodule for the associative algebra $A$, where the right action $\cdot\colon M\times A \to M$ is defined by
$m \cdot a = (-1)^{|a|\cdot |m|} a\cdot m$
 for homogeneous $a\in A$, $m\in M$.
\end{Remark}

As any supercommutative algebra has the structure of an associative algebra, we can study it with the Hochschild cohomology. Let's recall the definition here.

\begin{Definition}
 Let $A$ be an associative algebra, and $M$ an $A$-bimodule. For each $n\in \N$, define the \textit{$n$-th Hochschild cochain}
 \[\widehat{C}^n(A,M)=\bigl\{\C\text{-linear map }f\colon A^{\otimes n}\to M\bigr\}. \]
 Define the \textit{coboundary map} $\pd\colon \widehat{C}^n\to \widehat{C}^{n+1}$ by
 \begin{align*}
 & (\pd f )(a_1\otimes\dots\otimes a_{n+1})\\
 &\qquad = a_1f(a_2\otimes\dots\otimes a_{n+1})
 +\sum_{i=1}^n (-1)^i f(a_1 \otimes \dots\otimes a_i\cdot a_{i+1}\otimes\dots \otimes a_{n+1}) \\
 & \qquad\quad {}+ (-1)^{n+1}f(a_1\otimes\dots\otimes a_n)a_{n+1}.
 \end{align*}
\end{Definition}

\begin{Theorem}\label{HochschildThm}
 For all $n\in\N$, $\pd^2 \widehat{C}^n=0$ $($or simply $\pd^2=0)$. Therefore, the collection $(\widehat{C}^n,\pd)_{n\in \N}$ forms a cochain complex, called the \textit{Hochschild cochain complex}.
\end{Theorem}

\begin{proof}
 A conceptual proof can be found in \cite{Weibel} using simplicial methods. We shall not elaborate on the details here.
\end{proof}

\begin{Definition}
 Let $\widehat{Z}^n(A,M)=\ker(\pd)\subset \widehat{C}^n(A, M)$ be the space of Hochschild cocycles, $\widehat{B}^n(A,M)=\im(\pd)\subset \widehat{C}^n(A, M)$ be the space of Hochschild coboundaries. By Theorem~\ref{HochschildThm}, $\widehat{B}^n(A,M)\sseq \widehat{Z}^n(A,M)$. The quotient $\widehat{H}^n(A,M)=\widehat{Z}^n(A,M)/\widehat{B}^n(A,M)$ is the \textit{$n$-th Hochschild cohomology} of $A$ with respect to $M$.
\end{Definition}

\subsection{Shuffles and the super shuffle sum}
Recall the Harrison cohomology for a commutative associative algebra is defined by a subcochain complex of the Hochschild cochain complex. Using the same idea, we shall now give a cohomology theory of supercommutative algebras. To start, we need to define shuffles and the shuffle sums.

\begin{Definition}
 Let $n\in \Z_+$ and fix $p\in \Z$, $1 \leq p\leq n-1$. A permutation $\sigma\in \Sym(1, \dots , n)$ is a \textit{shuffle} of the letters $1, \dots , n$ if
 \begin{align}
 \sigma(1)<\dots<\sigma(p) \qquad \text{and}\qquad \sigma(p+1)<\dots<\sigma(n). \label{shuffle-p}
 \end{align}
 Let $J_p(1, \dots , n)$ be the collection of all shuffles satisfying (\ref{shuffle-p}). Clearly, a shuffle $\sigma\in J_p(1, \dots , n)$ may be identified with two increasing sequences respectively with lengths $p$ and $n-p$, consisting of distinct numbers in $\{1, \dots , n\}$. In particular,
 \[| J_p(1, \dots , n)| = \binom{n}{p}. \]
\end{Definition}
\begin{Remark}
 When taking the sum over all shuffles in $J_p(1, \dots , n)$, it is important to note that, unlike how the summation of $\sigma\in \Sym(1, \dots ,n)$ can be interchanged with the summation of $\sigma^{-1}\in \Sym(1, \dots , n)$, we cannot replace $\sigma$ in the summation of $\sigma \in J_p(1, \dots , n)$ by $\sigma^{-1}$. Generally, $\sigma\in J_p(1, \dots , n)$ does not imply $\sigma^{-1}\in J_p(1, \dots , n)$. For example, consider $n = 5$ and
 \[\sigma = \begin{pmatrix}
 1 & 2 & 3 & 4 & 5\\
 2 & 4 & 5 & 1 & 3
 \end{pmatrix}.\]
 Then $\sigma\in J_3(1, \dots , 5)$. But
 \[\sigma^{-1} = \begin{pmatrix}
 1 & 2 & 3 & 4 & 5\\
 4 & 1 & 5 & 2 & 3
 \end{pmatrix}\]
 is no longer an element in $J_p(1, \dots , n)$ for any $p$.
\end{Remark}
Recall that $A^{\otimes n}$ admits a right action of the symmetric group $\Sym(1,\dots ,n)$, defined by
\[\sigma(a_1\otimes \cdots \otimes a_n) = a_{\sigma(1)}\otimes \cdots \otimes a_{\sigma(n)}.\]
This right action induces a left $\Sym(1,\dots ,n)$-action on linear maps $f\colon A^{\otimes n}\to M$, given by
\[(\sigma f)(a_1 \otimes \cdots \otimes a_n) = f(a_{\sigma(1)}\otimes \cdots \otimes a_{\sigma(n)}). \]

\begin{Definition}
Fix homogeneous $a_1,\dots, a_n\in A$ where the odd elements are $a_{\alpha_1}, \dots, a_{\alpha_k}$ for $1\leq \alpha_1 < \cdots < \alpha_k \leq n$. For a fixed $\sigma\in \text{Sym}(1,\dots,n)$, let $1\leq \beta_1 < \cdots < \beta_k \leq n$ such that $a_{\sigma(\beta_1)}, \dots , a_{\sigma(\beta_k)}$ are odd. Then $\sigma(\beta_1), \dots , \sigma(\beta_k)$ is a permutation of $\{\alpha_1, \dots , \alpha_k\}$. We define $\sigma^o \in \text{Sym}(\alpha_1, \dots , \alpha_k)$ by
\[\sigma^o (\alpha_m) = \sigma(\beta_m), \qquad m = 1, \dots , k. \]
For convenience, we shall refer to $\alpha_1, \dots , \alpha_k$ as the \textit{odd indices} of the sequence $(1, \dots , n)$, and $\sigma(\beta_1), \dots , \sigma(\beta_k)$ as the \textit{odd indices} of the sequence $(\sigma(1), \dots , \sigma(n))$. \label{odd-element-definition}
\end{Definition}
\begin{Remark}\label{Rema-sigma-o}\quad
\begin{enumerate}\itemsep=0pt
\item[1.] Conceptually, $\sigma^o$ keeps track of the change of ordering of the odd elements. In particular, if we consider the product $a_1 \cdots a_n$ in $A$, then
\[a_{\sigma(1)} \cdots a_{\sigma(n)} = (-1)^{\sigma^o} a_1 \cdots a_n, \]
where $(-1)^{\sigma^o}$ is the signature of $\sigma^o \in \text{Sym}(\alpha_1, \dots , \alpha_k)$.
\item[2.] 
Care is needed with repeated application of the $^o$ operation, as each new instance of~$^o$ involves an implicit reordering of the input indices. Therefore, for $\sigma, \tau\in S_n$, $(\sigma\tau)^o$~generally does not coincide with~$\sigma^o \tau^o$. Below are examples of $\sigma$ and $\tau$ with $\sigma \tau = {\rm id}$, but $\sigma^o \tau^o \neq (\sigma\tau)^o = {\rm id}$.
\item[3.] 
Let
\[\sigma = \begin{pmatrix}
 1 & 2 & 3 & \cdots & n-1 & n & n+1\\
 n+1 & 1 & 2 & \cdots & n-2 & n-1 & n
\end{pmatrix}.\]
Then for $a_1, \dots , a_{n+1}\in A$ homogeneous, we have
\[(-1)^{\sigma^o} = (-1)^{(|a_1| + \cdots + |a_{n}|)\cdot |a_{n+1}|}.\]
\item[4.] 
Let
 \[\tau = \begin{pmatrix}
 1 & 2 & 3 & \cdots & n-1 & n & n+1\\
 2 & 3 & 4 & \cdots & n & n+1 & 1
 \end{pmatrix},\]
 namely, the inverse of $\sigma$ studied in~(3). Then for $a_1, \dots , a_n \in A$ homogeneous, we have
 \[(-1)^{\tau^o} = (-1)^{(|a_2|+\cdots +|a_{n+1}|)\cdot |a_1|}. \]
\item[5.] If $\sigma^o\tau^o = {\rm id}$, then $(-1)^{\sigma^o}(-1)^{\tau^o}$ should be identically 1 regardless of parities of $a_1, \dots , a_n$. This is obviously not the case, as we may choose $|a_1| = 0$, $|a_2| + \cdots + |a_{n}| = 1$, and $|a_{n+1}|=1$.
\end{enumerate}

\end{Remark}

\begin{Definition}
Fix positive integers $n$ and $p\leq n-1$. For a linear map $f\colon A^{\otimes n}\to M$, we define the \textit{super shuffle sum}
\[su_{n,p} f\colon \ A^{\otimes n} \to M\]
by
\begin{align*}
 (su_{n,p} f)(a_1\otimes \cdots \otimes a_n) = \sum_{\sigma\in J_p(1, \dots , n)}(-1)^\sigma (-1)^{(\sigma\iv)^o} f(a_{\sigma\iv(1)}\otimes \cdots \otimes a_{\sigma\iv(n)}).
\end{align*}
where $a_1, \dots , a_n$ are $\Z_2$-homogeneous. The definition naturally extends to $A^{\otimes n}$.
\end{Definition}

\begin{Remark}
 Note that $(-1)^\sigma$ is the signature of $\sigma\in S_n$ that is independent of the parity of $a_1, \dots , a_n$, while $(-1)^{(\sigma^{-1})^o}$ depends both on $\sigma$ and the parity of $a_1, \dots , a_n$.
\end{Remark}

\subsection{The super Harrison cochain}

\begin{Definition}\label{superHarrison-defn}
 Let $A$ be a supercommutative algebra.
 \begin{enumerate}\itemsep=0pt
 \item[1.] For each $n\in \N$, the tensor product $A^{\otimes n}$ is also a superspace with parity given by
 \[|a_1\otimes \cdots \otimes a_n| = |a_1| + \cdots + |a_n| \in \Z_2\]
 for homogeneous $a_1, \dots , a_n\in A$.
 \item[2.] Let $M$ be an $A$-supermodule. A linear map $f\colon A^{\otimes n} \to M$ is \textit{parity-preserving} if
 \[|f(a_1\otimes \cdots \otimes a_n)| = |a_1| + \cdots + |a_n|\in \Z_2.\]
 \item[3.] For each $n\in \N$, define the $n$-th \textit{super-Harrison cochain} by
 \[C^n(A, M) = \{f\colon A^{\otimes n} \to M \mid f \text{ is parity-preserving, and }su_{n,p}f = 0,\, p = 1, \dots , n-1 \}.\]
 Clearly, $C^n(A, M)$ is a subset of the Hochschild cochain $\widehat{C}^n(A, M)$.
 \end{enumerate}
\end{Definition}

\begin{Theorem}\label{superHarrisonthm}
 With all notations as in Definition~{\rm \ref{superHarrison-defn}}, for every $f\in C^n(A,M)$,
 \[\partial f \in C^{n+1}(A,M).\]
\end{Theorem}

\begin{proof}
 Fix a positive integer $p\leq n$. We analyze the super shuffle sum $su_{n+1, p}(\pd f)$, namely,
 \begin{align}
 & \sum_{\sigma\in J_p(1, \dots , n+1)}(-1)^\sigma (-1)^{(\sigma\iv)^o}(\pd f)(a_{\sigma^{-1}(1)}\otimes\dots\otimes a_{\sigma^{-1}(n+1)}) \nonumber\\
&\quad{} = \sum_{\sigma\in J_p(1, \dots , n+1)}(-1)^\sigma (-1)^{(\sigma\iv)^o}a_{\sigma\iv(1)}f\left(a_{\sigma\iv(2)}\otimes\dots\otimes a_{\sigma\iv(n+1)}\right) \label{shuffle-sum-of-df-1}\\
 & \quad \quad{}+\sum_{\sigma\in J_p(1, \dots , n+1)}(-1)^\sigma (-1)^{(\sigma\iv)^o}(-1)^{n+1}f\left(a_{\sigma\iv(1)}\otimes\dots\otimes a_{\sigma\iv(n)}\right)a_{\sigma\iv(n+1)}\label{shuffle-sum-of-df-2}\\
 & \quad \quad{}+\sum_{i=1}^n(-1)^i\sum_{\sigma\in J_p(1, \dots , n+1)}(-1)^\sigma (-1)^{(\sigma\iv)^o} \nonumber \\
 & \quad \quad{}\times f(a_{\sigma^{-1}(1)}\otimes\cdots\otimes a_{\sigma\iv(i)}\cdot a_{\sigma\iv(i+1)}\otimes\dots\otimes a_{\sigma^{-1}(n+1)})\label{shuffle-sum-of-df-3}.
 \end{align}
 We shall proceed to show that both $(\ref{shuffle-sum-of-df-1}) + (\ref{shuffle-sum-of-df-2})$ and (\ref{shuffle-sum-of-df-3}) all vanish.

\textbf{The analysis of (\ref{shuffle-sum-of-df-1}).}
 Note that $\sigma\in J_p(1, \dots , n+1)$. Then the number 1 sits either in the first increasing sequence or the second. This means either $\sigma(1) = 1$ or $\sigma(p+1)=1$. Therefore, either $\sigma^{-1}(1)=1$ or $\sigma^{-1}(1) = p+1$.

1. If $\sigma^{-1}(1)=1$, then the increasing sequences
 \[\sigma(2) < \cdots < \sigma(p), \sigma(p+1) < \cdots < \sigma(n+1)\]
 form a shuffle $\tau\in J_{p-1}(2, \dots , n+1)$, with $\tau(m)= \sigma(m), m = 2, \dots , n+1$. Clearly, the extension of $\tau$ in $\Sym(1, \dots , n+1)$ (obtained by supplementing $\tau(1) = 1$) coincides with~$\sigma$. Therefore, we have
$(-1)^\tau = (-1)^\sigma, (-1)^{\tau^o} = (-1)^{\sigma^o}$.
 Thus
 \begin{align}
 & \sum_{\substack{\sigma\in J_p(1, \dots , n+1)\\ \sigma^{-1}(1)=1}}(-1)^\sigma (-1)^{(\sigma\iv)^o}a_{\sigma\iv(1)}f (a_{\sigma\iv(2)}\otimes\dots\otimes a_{\sigma\iv(n+1)} ) \nonumber\\
 &\qquad = a_1 \cdot \sum_{\substack{\tau\in J_{p-1}(2, \dots , n+1)}}(-1)^\tau (-1)^{(\tau\iv)^o}f (a_{\tau\iv(2)}\otimes\dots\otimes a_{\tau\iv(n+1)} ). \label{shuffle-sum-of-df-1-1}
 \end{align}
 But since $f\in C^n(A, M)$, we know that $su_{n,p-1} f = 0$ in case $p\geq 2$. Evaluated at $a_2 \otimes \cdots \otimes a_{n+1}$, we see that
 \[\sum_{\substack{\tau\in J_{p-1}(2, \dots , n+1)}}(-1)^\tau (-1)^{(\tau\iv)^o}f\left(a_{\tau\iv(2)}\otimes\dots\otimes a_{\tau\iv(n+1)}\right) = 0.\]
 Therefore, (\ref{shuffle-sum-of-df-1-1}) vanishes if $p\geq 2$.

2. If $\sigma^{-1}(1)=p+1$, then the increasing sequences
 \[\sigma(1) < \cdots < \sigma(p),\qquad \sigma(p+2) < \cdots < \sigma(n+1)\]
 form a shuffle $\tau\in J_{p}(1,\dots , \widehat{p+1},\dots , n+1)$, with
 \[\tau(m)= \begin{cases}
 \sigma(m)-1 & \text{if }2\leq\sigma(m)\leq p+1,\\
 \sigma(m) & \text{if }\sigma(m) \geq p+2.
 \end{cases}\]
 To calculate the signature $(-1)^\tau$, let $\tilde{\tau}$ be the extension of $\tau$ in $\Sym(1, \dots , n+1)$, i.e.,
 \[\tilde{\tau}(m) = \begin{cases}
 \tau(m) & \text{if }m \neq p+1,\\
 p+1 & \text{if }m = p+1.
 \end{cases}\]
 Then $(-1)^\tau = (-1)^{\tilde{\tau}}$. Moreover, we have
$\sigma = (1, 2, \dots , p+1)\tilde{\tau}$.
 Therefore,
 \[(-1)^\sigma = (-1)^p (-1)^{\tau}=(-1)^p (-1)^{\tilde{\tau}}.\]
 To calculate the signature $(-1)^{\tau^o}$, let $1\leq \alpha_1 < \cdots < \alpha_k \leq n+1$ be the indices of the odd elements in the sequence $(a_1, \dots , a_{n+1})$. Let $1\leq \beta_1 < \cdots < \beta_k \leq {n+1}$ be the indices of the odd elements in the sequence $(a_{\sigma^{-1}(1)}, \dots , a_{\sigma^{-1}(n+1)})$.

$\bullet$ If $a_{p+1}$ is even, then $p+1$ appears neither in the sequence $\alpha_1, \dots , \alpha_k$ nor $\sigma^{-1}(\beta_1), \dots ,\allowbreak \sigma^{-1}(\beta_k)$. From the relation
$\sigma^{-1} = \tilde{\tau}^{-1}(p+1, p, \dots , 2, 1)$,
 we know that
 \[\sigma^{-1}(\beta_m) = \begin{cases}
 \tau^{-1}(\beta_m-1) & \text{if }2\leq\beta_m \leq p+1,\\
 \tau^{-1}(\beta_m) & \text{if }\beta_m \geq p+2.
 \end{cases}\]
 In other words, only the $\beta_m$'s before $p+1$ are shifted by 1; the others remain the same. Thus, the odd elements in $(a_{\tau^{-1}(1)}, \dots, a_{\tau^{-1}(p-1)}, a_{\tau^{-1}(p+1)},\dots, a_{\tau^{-1}(n+1)})$ remain to be $(a_{\sigma^{-1}(\beta_1)}, \dots,\allowbreak a_{\sigma^{-1}(\beta_{n+1})})$. We may conclude that $(\sigma^{-1})^o = (\tau^{-1})^o$. In particular, $(-1)^{\sigma^o} = (-1)^{\tau^o}$. Therefore, we conclude that
 \begin{align}
 & \sum_{\substack{\sigma\in J_p(1, \dots , n+1)\\ \sigma^{-1}(1)=p+1}}(-1)^\sigma (-1)^{(\sigma\iv)^o}a_{\sigma\iv(1)}f (a_{\sigma\iv(2)}\otimes\dots\otimes a_{\sigma\iv(n+1)} ) \nonumber\\
&\qquad{} = (-1)^p a_{p+1} \cdot \sum_{\substack{\tau\in J_p(1, \dots , \widehat{p+1}, \dots , n+1)}}(-1)^\tau (-1)^{(\tau\iv)^o} \nonumber \\
 & \quad \qquad{} \times f (a_{\tau\iv(1)}\otimes\cdots\otimes \reallywidehat{a_{\tau^{-1}(p+1)}} \otimes \cdots \otimes a_{\tau\iv(n+1)} ). \label{shuffle-sum-of-df-1-2}
 \end{align}
 Since $f\in C^n(A, M)$, we know that $su_{n,p}f = 0$ in case $p\leq n-1$. Evaluated at $a_{\tau\iv(1)}\otimes\dots \otimes \reallywidehat{a_{\tau^{-1}(p+1)}} \otimes \cdots \otimes a_{\tau\iv(n+1)}$, we see that
 \[\sum_{\substack{\tau\in J_p(1, \dots , \widehat{p+1}, \dots , n+1)}}(-1)^\tau (-1)^{(\tau\iv)^o}f (a_{\tau\iv(1)}\otimes\dots \otimes \reallywidehat{a_{\tau^{-1}(p+1)}} \otimes \cdots \otimes a_{\tau\iv(n+1)} ) = 0.\]
 Therefore, (\ref{shuffle-sum-of-df-1-2}) vanishes if $p\leq n-1$.

$\bullet$ If $a_{p+1}$ is odd, then $\beta_1=1$ and $\sigma^{-1}(\beta_1) = p+1$. Let $q = \# \{1\leq j \leq p+1 \mid a_j \text{ is odd}\}$.
 Then $\alpha_q = p+1$. So
 \[\bigl(\sigma^{-1}\bigr)^o = \begin{pmatrix}
 \alpha_1 & \alpha_2 & \cdots & \alpha_{q-1} & \alpha_q = p+1 & \alpha_{q+1} & \cdots & \alpha_k \\
 p+1 & \sigma^{-1}(\beta_2) & \cdots &\sigma^{-1}(\beta_{q-1}) & \sigma^{-1}(\beta_{q}) & \sigma^{-1}(\beta_{q+1}) & \cdots & \sigma^{-1}(\beta_k)
 \end{pmatrix}.\]
 With the same analysis as above, we see that
 \[\bigl(\tau^{-1}\bigr)^o = \begin{pmatrix}
 \alpha_1 & \cdots & \alpha_{q-1} & \alpha_{q+1} & \cdots & \alpha_k \\
 \sigma^{-1}(\beta_2) & \cdots &\sigma^{-1}(\beta_{q-1}) & \sigma^{-1}(\beta_{q}) & \cdots & \sigma^{-1}(\beta_{k}) \end{pmatrix}.\]
 Therefore, we conclude that (\ref{shuffle-sum-of-df-1-2}) equals to
 \begin{align*}
 &(-1)^p a_{p+1}\cdot (-1)^{q-1} \cdot \sum_{\substack{\tau\in J_p(1, \dots , \widehat{p+1}, \dots , n+1)}}(-1)^\tau (-1)^{(\tau\iv)^o}\\
 & \qquad {} \times f (a_{\tau\iv(1)}\otimes\cdots \otimes \reallywidehat{a_{\tau^{-1}(p+1)}} \otimes \cdots \otimes a_{\tau\iv(n+1)} ).
 \end{align*}
 By a similar argument as above, we see that (\ref{shuffle-sum-of-df-1-2}) vanishes if $p \leq n-1$.

3. We have handled most parts of (\ref{shuffle-sum-of-df-1}) except for the case where $p=1$ and the part with $\sigma^{-1}(1)=1$, and the case where $p = n$ and the part with $\sigma^{-1}(1)=p+1=n+1$. To process these terms, we will need the corresponding contributions from~(\ref{shuffle-sum-of-df-2}).

\textbf{The analysis of $(\ref{shuffle-sum-of-df-1})+(\ref{shuffle-sum-of-df-2})$.}
 With a similar procedure as in the analysis of (\ref{shuffle-sum-of-df-1}), we may show that the parts of (\ref{shuffle-sum-of-df-2}) where $p \leq n-1$, $\sigma^{-1}(n+1) = n+1$ and $p\geq 2$, $\sigma^{-1}(n+1)=p$ all vanish. What remains are the case with $p=n$ and the part with $\sigma^{-1}(n+1) = n+1$, and the case with $p=1$ and the part with $\sigma^{-1}(n+1) = 1$. Fix $a_1, \dots , a_{n+1}\in A$ homogeneous and assume the odd elements are $a_{\alpha_1}, \dots , a_{\alpha_k}$ for $1\leq \alpha_1 < \cdots < \alpha_k \leq n+1$. We first combine the remainder terms of (\ref{shuffle-sum-of-df-1}) and (\ref{shuffle-sum-of-df-2}) in the case of $p=1$.

1. If $p=1$ and $\sigma^{-1}(1)=1$, this means $\sigma = {\rm id}$. Thus, $\sigma^o= {\rm id}$. The corresponding term in~(\ref{shuffle-sum-of-df-1}) is
 \begin{align}
 a_1 f(a_2 \otimes \cdots \otimes a_{n+1}). \label{shuffle-sum-of-df-1-b1}
 \end{align}

2. If $p=1$ and $\sigma^{-1}(n+1) = 1$, this means
 \[\sigma^{-1} = \begin{pmatrix}
 1 & 2 & \cdots & n & n+1 \\
 2 & 3 & \cdots & n+1 & 1
 \end{pmatrix} \implies (-1)^\sigma = (-1)^n.\]
 We have analyzed $(\sigma^{-1})^o$ in Remark~\ref{Rema-sigma-o}\,(4) and concluded that
$(-1)^{(\sigma\iv)^o}\! = (-1)^{(k-1)\cdot |a_{1}|}$.
 The corresponding term in~(\ref{shuffle-sum-of-df-2}) is
 \begin{align}
 -(-1)^{(|a_2| + \cdots + |a_{n+1}|)\cdot |a_{1}|} f(a_2\otimes \cdots \otimes a_{n+1})a_{1}. \label{shuffle-sum-of-df-2-b2}
 \end{align}
 Note that since $M$ is a supermodule and $f$ preserves parity, from Remark~\ref{rema-right-action}, we know that~(\ref{shuffle-sum-of-df-2-b2}) is precisely the negative of~(\ref{shuffle-sum-of-df-1-b1}). Thus they sum up to be zero.

In the case of $p=n$, we similarly combine the remainder terms of~(\ref{shuffle-sum-of-df-1}) and~(\ref{shuffle-sum-of-df-2}).

1. If $p=n$ and $\sigma(n+1) = 1$, this means
 \[\sigma^{-1} = \begin{pmatrix}
 1 & 2 & \cdots & n & n+1 \\
 n+1 & 1 & \cdots & n-1 & n
 \end{pmatrix} \implies (-1)^\sigma = (-1)^n.\]
 We have analyzed $(\sigma^{-1})^o$ in Remark~\ref{Rema-sigma-o}\,(3) and concluded that
 \[(-1)^{(\sigma\iv)^o} = (-1)^{(k-1)\cdot |a_{n+1}|}.\]
 The corresponding term in~(\ref{shuffle-sum-of-df-1}) is
 \begin{align}
 (-1)^{n}(-1)^{(|a_1| + \cdots + |a_n|)\cdot |a_{n+1}|} a_{n+1} f(a_1 \otimes \cdots \otimes a_n). \label{shuffle-sum-of-df-1-b2}
 \end{align}

2. If $p=n$ and $\sigma(n+1) = n+1$, this means $\sigma = {\rm id}$. Thus, $\sigma^o = {\rm id}$. The corresponding term in~(\ref{shuffle-sum-of-df-2}) is
 \begin{align}
 (-1)^{n+1} f(a_1\otimes \cdots \otimes a_n)a_{n+1}. \label{shuffle-sum-of-df-2-b1}
 \end{align}
 Again from the assumption of $f$ and Remark~\ref{rema-right-action}, we know that (\ref{shuffle-sum-of-df-1-b2}) is precisely the negative of (\ref{shuffle-sum-of-df-2-b1}). Thus they sum up to be zero.

\textbf{The analysis of (\ref{shuffle-sum-of-df-3}).}
 The sum (\ref{shuffle-sum-of-df-3}) may be decomposed as
 \begin{align}
 & \sum_{i=1}^n (-1)^i \sum_{\substack{\sigma\in J_p(1, \dots , n+1) \\ 1\leq \sigma^{-1}(i) \leq p\\ 1\leq \sigma^{-1}(i+1) \leq p}} (-1)^\sigma (-1)^{(\sigma\iv)^o}\nonumber\\
 &\qquad{}\times f (a_{\sigma^{-1}(1)}\otimes \dots\otimes a_{\sigma\iv(i)}\cdot a_{\sigma\iv(i+1)}\otimes\dots \otimes a_{\sigma^{-1}(n+1)} )\label{shuffle-sum-of-df-3-1}\\
 & + \sum_{i=1}^n (-1)^i \sum_{\substack{\sigma\in J_p(1, \dots , n+1) \\ 1\leq \sigma^{-1}(i)\leq p \\
 p+1\leq \sigma^{-1}(i+1) \leq n+1}} (-1)^\sigma (-1)^{(\sigma\iv)^o}\nonumber\\
 &\qquad{}\times f (a_{\sigma^{-1}(1)}\otimes \cdots\otimes a_{\sigma\iv(i)}\cdot a_{\sigma\iv(i+1)}\otimes\dots\otimes a_{\sigma^{-1}(n+1)} )\label{shuffle-sum-of-df-3-2}\\
 & + \sum_{i=1}^n (-1)^i \sum_{\substack{\sigma\in J_p(1, \dots , n+1) \\ 1\leq \sigma^{-1}(i+1) \leq p \\
 p+1\leq \sigma^{-1}(i) \leq n+1}} (-1)^\sigma (-1)^{(\sigma\iv)^o}\nonumber\\
 &\qquad{}\times f (a_{\sigma^{-1}(1)}\otimes \cdots\otimes a_{\sigma\iv(i)}\cdot a_{\sigma\iv(i+1)}\otimes\dots \otimes a_{\sigma^{-1}(n+1)} )\label{shuffle-sum-of-df-3-3}\\
 & + \sum_{i=1}^n (-1)^i \sum_{\substack{\sigma\in J_p(1, \dots , n+1) \\ p+1\leq \sigma^{-1}(i) \leq n+1\\ p+1 \leq \sigma^{-1}(i+1) \leq n+1}} (-1)^\sigma (-1)^{(\sigma\iv)^o}\nonumber\\
 &\qquad{}\times f (a_{\sigma^{-1}(1)}\otimes \cdots\otimes a_{\sigma\iv(i)}\cdot a_{\sigma\iv(i+1)}\otimes\dots\otimes a_{\sigma^{-1}(n+1)} )\label{shuffle-sum-of-df-3-4}
 \end{align}
 according to the positions of $\sigma^{-1}(i)$ and $\sigma^{-1}(i+1)$.

1. We first analyze (\ref{shuffle-sum-of-df-3-1}). For a fixed $i$ between $1$ and $n$, the increasing sequence corresponding to $\sigma$ looks like
 \begin{gather*}
 1\leq \sigma(1) < \cdots < i < i+1 < \cdots < \sigma(p) \leq n+1, \\
 1 \leq \sigma(p+1) < \cdots < \sigma(n+1) \leq n+1.
 \end{gather*}
 In particular, there exists an integer $1\leq q \leq p-1$ such that $q = \sigma^{-1}(i)$, $q+1 = \sigma^{-1}(i+1)$. From the property of shuffles, we know that $q \leq i$. Then (\ref{shuffle-sum-of-df-3-1}) may be rewritten as
 \begin{align}
 & \sum_{q=1}^{p-1} \sum_{\substack{\sigma\in J_p(1, \dots , n+1)\\ \sigma^{-1}(1)=q, \sigma^{-1}(2) = q+1}} (-1)^1 (-1)^\sigma (-1)^{(\sigma\iv)^o} f\left(a_q a_{q+1} \otimes a_{\sigma^{-1}(3)} \otimes \cdots \otimes a_{\sigma^{-1}(n+1)}\right) \label{shuffle-sum-of-df-3-1-1}\\
 & + \cdots + \sum_{q=1}^{p-1} \sum_{\substack{\sigma\in J_p(1, \dots , n+1)\\ \sigma^{-1}(i)=q, \sigma^{-1}(i+1) = q+1}} (-1)^i (-1)^\sigma (-1)^{(\sigma\iv)^o}\nonumber\\
 &\hphantom{+ \cdots +}{}\times
 f (a_{\sigma^{-1}(1)} \otimes \cdots \otimes a_q a_{q+1} \otimes \cdots \otimes a_{\sigma^{-1}(n+1)} ) \label{shuffle-sum-of-df-3-1-2}\\
 & + \cdots + \sum_{q=1}^{p-1} \sum_{\substack{\sigma\in J_p(1, \dots , n+1)\\ \sigma^{-1}(n)=q, \sigma^{-1}(n+1) = q+1}} (-1)^n (-1)^\sigma (-1)^{(\sigma\iv)^o}\nonumber\\
 &\hphantom{+ \cdots +}{}\times
 f (a_{\sigma^{-1}(1)} \otimes \cdots \otimes a_{\sigma^{-1}(n-1)} \otimes a_q a_{q+1} ). \label{shuffle-sum-of-df-3-1-3}
 \end{align}
 Set
 \[b_1 = a_1, \ \dots , \ b_{q-1} = a_{q-1}, \ b_q = a_q a_{q+1}, \ b_{q+1} = a_{q+2}, \ \dots , \ b_{n}= a_{n+1}.\]
 For fixed $i$ between $1$ and $n$, define the permutation $\tilde{\tau}\iv$ by requiring that
 \begin{align*}
 & f(a_{\sigma\iv(1)} \otimes \cdots \otimes a_{\sigma\iv(i-1)} \otimes a_{q}a_{q+1} \otimes a_{\sigma\iv(i+2)} \otimes \cdots \otimes a_{\sigma^{-1}(n+1)})\\
 &\qquad{} = f(b_{\tilde{\tau}\iv(1)} \otimes \cdots \otimes b_{\tilde{\tau}\iv(i-1)}\otimes b_{\tilde{\tau}\iv(i)}\otimes b_{\tilde{\tau}\iv(i+1)} \otimes \cdots \otimes b_{\tilde{\tau}^{-1}(n)})
 \end{align*}
 and
$\tilde{\tau}\iv(n+1)=n+1$.
 It is clear that
 \[\tilde{\tau}^{-1} = (n+1, n, \dots , q+2, q+1) \sigma^{-1} (i+1, i+2, \dots , n, n+1).\]
 Taking the inverse, we have
 \[\tilde{\tau} = (n+1, n, \dots , i+2, i+1) \sigma (q+1, q+2, \dots , n, n+1).\]
 In particular, if there exists a (necessarily unique) index $p+1 \leq r \leq n+1$ with $\sigma(r)<i, \sigma(r+1)>i+1$, then
 \begin{align*}
 & \tilde{\tau}(1)= \sigma(1), \ \dots , \ \tilde{\tau}(q) = \sigma(q) = i, \\
 & \tilde{\tau}(q+1) = \sigma(q+2)-1, \ \dots , \ \tilde{\tau}(p-1) = \sigma(p)-1 \\
 & \tilde{\tau}(p) = \sigma(p+1), \ \dots , \ \tilde{\tau}(r-1) = \sigma(r) < i, \\
 & \tilde{\tau}(r) = \sigma(r+1)-1>i, \ \dots , \ \tilde{\tau}(n) = \sigma(n+1) -1, \tilde{\tau}(n+1)=n+1.
 \end{align*}
 If we set $\tau$ as the restriction of $\tilde{\tau}$ on $\{1, \dots , n\}$, then clearly from $\sigma\in J_p(1, \dots , n+1)$, we know that
$\tau \in J_{p-1}(1, \dots , n)$.
 The summation over $\sigma\in J_p(1, \dots , n+1)$ with $\sigma^{-1}(i) = q$, $\sigma^{-1}(i+1)=q+1$ reduces to the summation of $\tau\in J_{p-1}(1, \dots , n)$ with $\tau^{-1}(i) = q$. Also, we have
 \[(-1)^\tau = (-1)^{\tilde{\tau}} = (-1)^{n-i} (-1)^\sigma (-1)^{n-q} \implies (-1)^i(-1)^\sigma = (-1)^{q}(-1)^\tau.\]
 To analyze $(\tau^{-1})^o$, let $1\leq \alpha_1 < \!\cdots\! < \alpha_k \leq n+1$ be the odd indices in the sequence of $a_1, \dots , a_{n+1}$, $1\leq \beta_1 < \cdots < \beta_k \leq n+1$ be the odd indices in the sequence of $a_{\sigma^{-1}(1)}, \dots , a_{\sigma^{-1}(n+1)}$.

$\bullet$ If $a_q$ and $a_{q+1}$ are both even, then $q$, $q+1$ do not appear in the odd indices of $(a_1, \dots , a_{n+1})$. Let $r$ be the index such that $\alpha_r<q$, $\alpha_{r+1}>q$. It is clear that $(\tau\iv)^o$ can be expressed as
 \begin{align}
 &\begin{pmatrix}
 \alpha_1 & \cdots & \alpha_r & \alpha_{r+1} & \cdots & \alpha_k\\
 \alpha_1 & \cdots & \alpha_r & \alpha_{r+1}-1 & \cdots & \alpha_k -1
 \end{pmatrix} \cdot \bigl(\sigma\iv\bigr)^o \nonumber\\
 &\qquad {} \times \begin{pmatrix}
 \alpha_1 & \cdots & \alpha_r & \alpha_{r+1} & \cdots & \alpha_k\\
 \alpha_1 & \cdots & \alpha_r & \alpha_{r+1}-1 & \cdots & \alpha_k -1
 \end{pmatrix}^{-1}.\label{conjugate}
 \end{align}
 In particular, $(-1)^{(\tau\iv)^o} = (-1)^{(\sigma\iv)^o}$.

$\bullet$ If $a_q$ is odd, $a_{q+1}$ is even, then $b_q$ is odd. In this case, we see that $q$ appears both in the odd indices of $a_1, \dots , a_{n+1}$ and $b_1, \dots , b_n$. Say $\alpha_r = q$. In this case, we see that $(\tau\iv)^o$ can also be expressed by (\ref{conjugate}). In particular, $(-1)^{(\tau\iv)^o} = (-1)^{(\sigma\iv)^o}$.

$\bullet$ If $a_q$ is even, $a_{q+1}$ is odd, then $b_q$ is odd. In this case, it suffices to supplement a conjugation by $(q,q+1)$ in (\ref{conjugate}). We still have $(-1)^{(\tau\iv)^o} = (-1)^{(\sigma\iv)^o}$.

$\bullet$ If both $a_q$ and $a_{q+1}$ are odd, say $\alpha_r = q$ and $\sigma^{-1}(\beta_s) = q$. Necessarily, $\alpha_{r+1} = q+1$, $\beta_s = i$, and $\sigma^{-1}(\beta_s+1) = q+1$. If $r=s$, the following is trivial, so assume that $r<s$. In this case, $\bigl(\sigma^{-1}\bigr)^o$ is the permutation given
 \[\begin{pmatrix}
 \alpha_1 & \cdots & \alpha_r = q & \alpha_{r+1} = q+1 & \cdots & \alpha_s & \alpha_{s+1} & \cdots & \alpha_k \\
 \sigma^{-1}(\beta_1) & \cdots & \sigma^{-1}(\beta_r) & \sigma^{-1}(\beta_{r+1}) & \cdots & q & q+1 & \cdots & \sigma^{-1}(\beta_k)
 \end{pmatrix}.\]
 To pass to $\bigl(\tau^{-1}\bigr)^o$, we need to remove both $q$ and $q+1$ from the domain and codomain of $\bigl(\sigma^{-1}\bigr)^o$, resulting in
 \begin{align*}
 &\left(\begin{matrix}
 \alpha_1 & \cdots & \alpha_{r-1} & \alpha_{r+2} & \cdots \\
 \sigma^{-1}(\beta_1) & \cdots & \sigma^{-1}(\beta_{r-1}) & \sigma^{-1}(\beta_{r}) & \cdots
 \end{matrix}\right.\\
 &\hspace{100pt}\left.\begin{matrix}
 \cdots & \alpha_{s} & \alpha_{s+1} & \alpha_{s+2} & \cdots & \alpha_k \\
 \cdots & \sigma^{-1}(\beta_{s-2}) & \sigma^{-1}(\beta_{s-1}) & \sigma^{-1}(\beta_{s+2}) & \cdots & \sigma^{-1}(\beta_k)
 \end{matrix}\right),
 \end{align*}
 whose signature is $(-1)^{r-s}\cdot (-1)^{r-s} (-1)^{(\sigma^{-1})^o} = (-1)^{(\sigma^{-1})^o}$. We then apply the conjugation as in (\ref{conjugate}) to obtain $(\tau^{-1})^o$. In particular, we still have $(-1)^{(\sigma\iv)^o} = (-1)^{(\tau\iv)^o}$.

 To summarize, we first rewrite (\ref{shuffle-sum-of-df-3-1}) as $(\ref{shuffle-sum-of-df-3-1-1})+ (\ref{shuffle-sum-of-df-3-1-2}) + (\ref{shuffle-sum-of-df-3-1-3})$, then showed that for each $i = 1, \dots , n$, the summation is
 \begin{align*}
 \sum_{q=1}^{p-1}\sum_{\tau\in J_{p-1}(1, \dots , n)} (-1)^q (-1)^\tau (-1)^{(\tau\iv)^o} f(b_{\tau^{-1}(1)} \otimes \cdots \otimes b_{\tau^{-1}(n)}).
 \end{align*}
 By the assumption that $su_{n,p-1} f(b_1 \otimes \cdots \otimes b_n) = 0$, we see that (\ref{shuffle-sum-of-df-3-1}) vanishes for every $p \geq 2$. Note that (\ref{shuffle-sum-of-df-3-1}) does not contain any parts with $p=1$. Therefore, we conclude that $(\ref{shuffle-sum-of-df-3-1}) = 0$.

2. We next analyze (\ref{shuffle-sum-of-df-3-4}). For a fixed $i$ between 1 and $n+1$, the increasing sequence corresponding to $\sigma$ looks like
 \begin{gather*}
 1\leq \sigma(1) < \cdots < \sigma(p) \leq n+1, \\
 1 \leq \sigma(p+1) < \cdots < i < i+1<\cdots < \sigma(n+1) \leq n+1.
 \end{gather*}
 We may now use a similar procedure as in the analysis of (\ref{shuffle-sum-of-df-3-1}), tracking the contracted term $a_qa_{q+1}$. In this case, we use the assumption that $su_{n,p}f(b_1\otimes\cdots\otimes b_n)=0$. Note that (\ref{shuffle-sum-of-df-3-4}) does not contain any parts with $p=n+1$. Thus, we conclude that $(\ref{shuffle-sum-of-df-3-4})=0$.

3. Finally, we analyze $(\ref{shuffle-sum-of-df-3-2})+(\ref{shuffle-sum-of-df-3-3})$. Consider a shuffle $\sigma$ appearing in (\ref{shuffle-sum-of-df-3-2}). Then for a~fixed $i$ between 1 and $n$, the increasing sequence corresponding to $\sigma$ looks like
 \begin{gather*}
 1\leq \sigma(1) < \cdots < i < \cdots < \sigma(p) \leq n+1, \\
 1 \leq \sigma(p+1) < \cdots < i+1 < \cdots < \sigma(n+1) \leq n+1.
 \end{gather*}
 Now we define the permutation $\tau\in\Sym(1,\dots ,n+1)$ by
 \[\tau(k) = \begin{cases}
 i & \text{if }\sigma(k)=i+1,\\
 i+1 & \text{if }\sigma(k)=i,\\
 \sigma(k) & \text{otherwise}.
 \end{cases}\]
 Then $\tau$ is also a shuffle in $J_p(1,\dots ,n+1)$, with increasing sequence
 \begin{gather*}
 1\leq \tau(1) < \cdots < i+1 < \cdots < \tau(p) \leq n+1, \\
 1 \leq \tau(p+1) < \cdots < i < \cdots < \tau(n+1) \leq n+1.
 \end{gather*}
 In particular, $\tau$ is a shuffle appearing in (\ref{shuffle-sum-of-df-3-3}). Furthermore, starting with some $\tau$ in (\ref{shuffle-sum-of-df-3-3}) we can build the corresponding $\sigma$ in (\ref{shuffle-sum-of-df-3-2}) with a symmetrical construction. In this way, terms in (\ref{shuffle-sum-of-df-3-2}) and (\ref{shuffle-sum-of-df-3-3}) can be paired together in a bijective way. Consider a fixed term in (\ref{shuffle-sum-of-df-3-2})
 \begin{equation}
 (-1)^i(-1)^\sigma(-1)^{(\sigma\iv)^o}f (a_{\sigma^{-1}(1)}\otimes \dots\otimes a_{\sigma\iv(i)}\cdot a_{\sigma\iv(i+1)}\otimes\dots \otimes a_{\sigma^{-1}(n+1)} )\label{sh-3-2-term}
 \end{equation}
 and its corresponding term in (\ref{shuffle-sum-of-df-3-3})
 \begin{align}
 &(-1)^i(-1)^\tau(-1)^{(\tau\iv)^o}f (a_{\tau^{-1}(1)}\otimes \dots\otimes a_{\tau\iv(i)}\cdot a_{\tau\iv(i+1)}\otimes\dots\otimes a_{\tau^{-1}(n+1)} )\nonumber\\
 &\qquad {}=(-1)^i(-1)^\tau(-1)^{(\tau\iv)^o}f (a_{\sigma^{-1}(1)}\otimes \dots\otimes a_{\sigma\iv(i+1)}\cdot a_{\sigma\iv(i)}\otimes\dots\otimes a_{\sigma^{-1}(n+1)} )\nonumber\\
& \qquad{} =(-1)^i(-1)^\tau(-1)^{(\tau\iv)^o}(-1)^{|a_{\sigma\iv(i)}||a_{\sigma\iv(i+1)}|}\nonumber\\
 &\qquad\quad{}\times f (a_{\sigma^{-1}(1)}\otimes \cdots\otimes a_{\sigma\iv(i)}\cdot a_{\sigma\iv(i+1)}\otimes\dots\otimes a_{\sigma^{-1}(n+1)} ),\label{sh-3-3-term}
 \end{align}
 where the supercommutativity of the algebra is used in (\ref{sh-3-3-term}). By construction
 \[\tau=(i,i+1)\sigma,\]
 so $(-1)^{\sigma}\cdot (-1)=(-1)^\tau$. Now we analyze $(\tau\iv)^o$. Let $1\leq \alpha_1 < \cdots < \alpha_k \leq n+1$ be the odd indices in the sequence of $a_1, \dots , a_{n+1}$, $1\leq \beta_1 < \cdots < \beta_k \leq n+1$ be the odd indices in the sequence of $a_{\sigma^{-1}(1)}, \dots , a_{\sigma^{-1}(n+1)}$.

$\bullet$ If $a_{\sigma\iv(i)}$ and $a_{\sigma\iv(i+1)}$ are both even, then neither appears in the odd indices $\beta_1,\dots,\beta_k$, so $(-1)^{(\sigma\iv)^o}=(-1)^{(\tau\iv)^o}$. Further, $(-1)^{|a_{\sigma\iv(i)}||a_{\sigma\iv(i+1)}|}=1$.

$\bullet$ If one of $a_{\sigma\iv(i)}$ and $a_{\sigma\iv(i+1)}$ are odd, then only one of the two appears in the odd indices $\beta_1,\dots,\beta_k$. Because the only difference between $\sigma$ and $\tau$ is the transposition interchanging $i$ and $i+1$, they will act identically on the odd indices, so once again $(-1)^{(\sigma\iv)^o}=(-1)^{(\tau\iv)^o}$. Again, $(-1)^{|a_{\sigma\iv(i)}||a_{\sigma\iv(i+1)}|}=1$.

$\bullet$ If both of $a_{\sigma\iv(i)}$ and $a_{\sigma\iv(i+1)}$ are odd, then $\sigma^o$ only differs from $\tau^o$ by a transposition, so $(-1)^{(\sigma\iv)^o}=(-1)^{(\tau\iv)^o+1}$. In this case, $(-1)^{|a_{\sigma\iv(i)}||a_{\sigma\iv(i+1)}|}=-1$.

 All these three cases show that
 \[(-1)^{(\sigma\iv)^o}=(-1)^{(\tau\iv)^o}(-1)^{|a_{\sigma\iv(i)}||a_{\sigma\iv(i+1)}|},\]
 so (\ref{sh-3-2-term}) and (\ref{sh-3-3-term}) only differ by a sign. Therefore, $(\ref{sh-3-2-term})+(\ref{sh-3-3-term})=0$. Note that no terms with $p+1=1$ or $p=n+1$ appear in (\ref{shuffle-sum-of-df-3-2}) or (\ref{shuffle-sum-of-df-3-3}). Thus, each of the terms of (\ref{shuffle-sum-of-df-3-2}) cancels with a corresponding term of (\ref{shuffle-sum-of-df-3-3}), so their sum vanishes.

 We have now shown that $(\ref{shuffle-sum-of-df-3})=0$, so $(\ref{shuffle-sum-of-df-1})+(\ref{shuffle-sum-of-df-2})$ and (\ref{shuffle-sum-of-df-3}) all vanish, proving that $\pd f\in C^{n+1}(A,M)$.
\end{proof}

\begin{Definition}
 With Theorem~\ref{superHarrisonthm}, we see that $C^n(A, M)$ forms a cochain complex, called the \textit{super Harrison cochain complex}. Let $Z^n(A, M) = \ker(\pd) \subset C^n(A, M)$ and $B^n(A, M) = \im (\pd) \subset C^n(A, M)$. We define the \textit{super Harrison cohomology} by
 \[H^n(A, M) = Z^n(A, M)/B^n(A, M).\]
\end{Definition}

\begin{Remark} \label{HarrisonRestriction}
 If $A$ is a commutative algebra (i.e., $A^1=0$) and $M$ an even module, the super Harrison cohomology agrees with the usual Harrison cohomology, since $(-1)^{(\sigma\iv)^o}=1$ restricted to $A^0$.
\end{Remark}

\begin{Remark}
 Let $A$ be a supercommutative algebra, $M$ an $A$-module, and $f\colon A^{\otimes 2}\to M$ a~parity preserving linear map. Let $a,b\in A$ be homogeneous elements. The condition $su_{2,1}f=0$ then is precisely
 \[f(a\otimes b)=(-1)^{|a||b|}f(b\otimes a).\]
 Thus within $C^2(A,M)$, the super-Harrison condition encodes that the maps from $A^{\otimes2}\to M$ must respect the supercommutativity of $A$, showing why this condition is natural to assume. In higher dimensions this condition generalizes the supercommutativity condition. For example, for a parity preserving linear map $g\colon A^{\otimes 3}\to M$, there are two shuffles to consider. For homogeneous elements $a,b,c\in A$, $su_{3,1}g=0$ states that
 \[g(a\otimes b\otimes c)-(-1)^{|a||b|}g(b\otimes a\otimes c)+(-1)^{|a|(|b|+|c|)}g(b\otimes c\otimes a)=0,\]
 and $su_{3,2}=0$ states that
 \[g(a\otimes b\otimes c)-(-1)^{|b||c|}g(a\otimes c\otimes b)+(-1)^{(|a|+|b|)|c|}g(c\otimes a\otimes b)=0.\]
\end{Remark}

\begin{Remark}
 With an argument similar to that for Harrison cohomology, we may show that $H^2(A, M)$ describes all the equivalence classes of square-zero extensions of $A$ by $M$. We may also show that $H^2(A, A)$ describes all the equivalence classes of first-order deformations of $A$ as supercommutative algebras. We shall, however, not elaborate the details for supercommutative algebra, but leave the space of the paper for the discussion of vertex superalgebras later.
\end{Remark}

\begin{Example}
 We demonstrate this by computing the second cohomology of the exterior algebra $A=\Lambda\bigl(\C^2\bigr)$, which has two generators $x$, $y$ and a basis $\{1,x,y,xy\}$. Recall $A$ has the multiplicative relations
 \begin{align*}
 xy=-yx,\qquad x^2=y^2=0.
 \end{align*}
 $A$ is a supercommutative algebra with the grading
 \begin{align*}
 A^0=\textrm{Span}\{1,xy\},\qquad A^1=\textrm{Span}\{x,y\}.
 \end{align*}
 \begin{Proposition}
 For $A=\Lambda\bigl(\C^2\bigr)$, $H^2(A,A)=0$.
 \end{Proposition}
 \begin{proof}
 Let $f\in Z^2(A,A)$. Because $f$ satisfies the super-shuffle condition, for homogeneous $a,b\in A$, $f(a\otimes b)=(-1)^{|a||b|}f(b\otimes a)$. As a result,
$f(x\otimes x)=f(y\otimes y)=0$.
 Note that the cochain restricts $f\colon A^{\otimes 2}\to A$ so that $f(a\otimes b)\in A^{|a|+|b|}$. Thus we can consider the matrix coefficients
 \begin{align*}
 &f(1\otimes 1)= \lambda_0+\lambda_1 xy, \qquad f(x\otimes y)=c_0+c_1xy,
 \end{align*}
 with $\lambda_0,\lambda_1,c_0,c_1\in\C$. As a~cocycle, $f$ is in the kernel of the Hochschild differential, so for $a,b,c\in A$,
 \begin{align*}
 (\pd f)(a\otimes b\otimes c)=a\cdot f(b\otimes c)-f(ab\otimes c)+f(a\otimes bc)-f(a\otimes b)\cdot c=0.
 \end{align*}
 We consider special cases of this. For $a\in A$,
 \begin{align*}
 &0=(\pd f)(1\otimes 1\otimes a)= f(1\otimes a)-f(1\otimes a)+f(1\otimes a)-f(1\otimes 1)a\\
 &\qquad \implies f(1\otimes a)= \lambda_0a+\lambda_1xya.
 \end{align*}
 Hence,
 \begin{align*}
 f(1\otimes x) = \lambda_0x,\qquad f(1\otimes y) = \lambda_0y,\qquad f(1\otimes xy) = \lambda_0xy.
 \end{align*}
 Next, consider
 \begin{align*}
 &0=(\pd f)(x\otimes x\otimes y)=xf(x\otimes y)- f(x^2\otimes y)+f(x\otimes xy)-f(x\otimes x)y\\
 &\qquad \implies x(c_0+c_1 xy)=-f(x\otimes xy).
 \end{align*}
 Using this and a similar calculation on $y\otimes y\otimes x$,
 \begin{align*}
 f(x\otimes xy)= -c_0 x,\qquad f(y\otimes xy)=-c_0y.
 \end{align*}
 Finally, consider
 \begin{align*}
& 0=(\pd f)(x\otimes y\otimes xy)=x f(y\otimes xy)-f(xy\otimes xy)+f(x\otimes yxy)- f(x\otimes y)xy\\
 &\qquad \implies x f(y\otimes xy)-f(x\otimes y)xy=f(xy\otimes xy)\\
 &\qquad \implies f(xy\otimes xy)=-2c_0 xy.
 \end{align*}
 This fully determines $f$ by showing how it acts on a basis of $A^{\otimes 2}$.

 Now consider $g\in C^1(A,A)$. As there is no super-shuffle condition to satisfy, $g$ is simply a~parity-preserving linear map $g\colon A\to A$. Thus we may define $g$ by
 \begin{align*}
 g(1)=\lambda _0+\lambda_1 xy,\qquad g(x)=g(y)=0,\qquad g(xy)=-c_0-c_1xy.
 \end{align*}
 For $a, b\in A$, $g$ gives rise to the 2-coboundary
 \[(\pd g)(a\otimes b)=a\cdot g(b)-g(ab)+g(a)\cdot b.\]
 From this, it is clear $(\pd g)(x\otimes x)=(\pd g)(y\otimes y)=0$. Further calculations show
 \begin{alignat*}{3}
& (\pd g)(1\otimes 1)=\lambda_0+\lambda_1xy,\qquad &&(\pd g)(x\otimes y)= c_0+c_1xy,\\
& (\pd g)(1\otimes x)= \lambda_0x,\qquad &&(\pd g)(x\otimes xy)=-c_0x,&\\
& (\pd g)(1\otimes y)=\lambda_0y,\qquad &&(\pd g)(y\otimes xy)= -c_0 y,&\\
& (\pd g)(1\otimes xy)= \lambda_0xy,\qquad &&(\pd g)(xy\otimes xy)= -2c_0 xy.&
 \end{alignat*}
 This exhausts how $(\pd g)$ acts on a basis of $A^{\otimes 2}$, and shows $(\pd g)=f$. Therefore, $f$ is a~2\nobreakdash-coboun\-dary, and vanishes in cohomology. Hence $Z^2(A,A)=B^2(A,A)$, and $H^2(A,A)=0$.
 \end{proof}

 This result makes sense when thought of in terms of deformations. $\Lambda\bigl(\C^2\bigr)$ is the free supercommutative algebra on two odd generators, and has no nontrivial supercommutative deformations. In comparison, the second Hochschild cohomology $\widehat{H}^2\bigl(\Lambda\bigl(\C^2\bigr),\Lambda\bigl(\C^2\bigr)\bigr)$ is six-dimensional, corresponding to deformations of the supercommutator. While these are valid deformations of \textit{associative} algebras, they destroy the supercommutative structure. Nonzero higher super-Harrison cohomology can show up in characteristic zero even in the purely commutative case: the ring $\C[x,y]/(xy)$ has a nontrivial 2-cocycle characterized by deforming the relation $xy=0$ to $xy=\lambda$, for a nonzero $\lambda\in \C$.
\end{Example}

\section{Cohomology theory for vertex superalgebras}\label{section3}

\subsection{Vertex superalgebras and their modules}
\begin{Definition}
 A {\it grading-restricted vertex superalgebra} is a $\frac{1}{2}\Z\times \Z_2$-graded $\C$-vector space \[V=\coprod_{n\in\frac{1}{2}\Z, \alpha\in \Z_2} V_{(n)}^\alpha=\coprod_{n\in\frac{1}{2}\Z}V_{(n)}^0\oplus \coprod_{n\in\frac{1}{2}\Z}V_{(n)}^1 = V^0 \oplus V^1,\] equipped with a {\it vertex operator map}
 \begin{align*}
 Y\colon \ V \to \Hom_\C{(V,V(\!(x)\!))},\qquad
 u \mapsto Y(u,x) = \sum_{n\in\Z}u_n x^{-n-1},
 \end{align*}
 and a vector \smash{$\vac\in V_{(0)}^0$} called the {\it vacuum}, satisfying the following axioms:
 \begin{enumerate}\itemsep=0pt
 \item[1.] {\it $\frac{1}{2}\Z$-grading conditions:} Let \smash{$V_{(n)} = V_{(n)}^0 \oplus V_{(n)}^1$} for each $n\in\frac{1}{2}\Z$. For every $n\in\frac{1}{2}\Z$, $\dim V_{(n)}<\infty$, and $V_{(n)} = 0$ if $n$ is sufficiently negative. Let $\db_V\colon V\to V$ be the $\frac{1}{2}\Z$\nobreakdash-grading operator, i.e., $\db_V v=nv$ for $v\in V_{(n)}$. Then for every $v\in V$, the following \textit{$\db$-commutator formula} holds:
 \[[\db_V,Y(v,x)]=x\frac{{\rm d}}{{\rm d}x}Y(v,x)+Y(\db_V v,x).\]
 \item[2.] {\it $\Z_2$-grading condition:} For $i,j\in \Z_2$, $u\in V^i$, $v \in V^j$, $Y(u, x)v \in V^{i+j}\bigl[\bigl[x, x^{-1}\bigr]\bigr]$. We shall say $u$, $v$ are $\Z_2$-homogeneous. We shall use the notation $|u| = i$, $|v| = j$ for the \textit{parity} of~$u$ and~$v$, respectively.
 \item[3.] {\it Identity and creation property:} For $v\in V$, $Y(\vac,x)v = v$, $Y(v,x)\vac\in v+xV[[x]]$.
 \item[4.] {\it $\Db$-derivative property and $\Db$-commutator formula:} Let $\Db_V\colon V\to V$ be the linear operator given by
 \[\Db_V v=\Res_xx^{-2}Y(v,x)\vac=v_{-2}\vac\]
 for $v\in V$. Then for $u\in V$,
 \[\frac{{\rm d}}{{\rm d}x}Y(u,x)=Y(\Db_V u,x)=[\Db_V,Y(u,x)].\]
 \item[5.] {\it Duality:} Let $u_1, u_2\in V$ be $\Z_2$-homogeneous. Then for every element $v\in V$ and $v'\in V' = \coprod_{n\in\frac{1}{2}\Z} V_n^*$ (the restricted dual of $V$), the following three series:
 \begin{gather*}
 \eval{Y(u_1,z_1)Y(u_2,z_2)}, \qquad
 (-1)^{|u_1||u_2|}\eval{Y(u_2,z_2)Y(u_1,z_1)},\\
 \eval{Y(Y(u_1,z_1-z_2)u_2,z_2)}
 \end{gather*}
 converge absolutely in the regions $|z_1|>|z_2|>0$, $|z_2|>|z_1|>0$, $|z_2|>|z_1-z_2|>0$, respectively, to a common rational function in $z_1$ and $z_2$ with the only possible poles at $z_1=0$, $z_2=0$ and $z_1=z_2$.
 \end{enumerate}
 We denote a grading-restricted vertex (super)algebra by $(V,Y,\vac)$, or just $V$.
\end{Definition}

A homomorphism of grading restricted vertex superalgebras $(V_1, Y_{V_1}, \vac_{V_1})$, $(V_2, Y_{V_2}, \vac_{V_2})$ is a~parity preserving linear map $f\colon V_1\to V_2$ such that for $u,v\in V_1$,
\begin{gather*}
 f(Y_{V_1}(u,z)v) =Y_{V_2}(f(u),z)f(v),\qquad
 f(\vac_{V_1}) =\vac_{V_2},\qquad
 f\circ\db_{V_1} =\db_{V_2}\circ f,\\
 f\circ\Db_{V_1} =\Db_{V_2}\circ f,
\end{gather*}
where in the first equality $f$ is extended linearly to $V_1(\!(z)\!)\to V_2(\!(z)\!)$.

\begin{Remark}
 The duality axiom is often referred to as vertex operators satisfying \textit{$($super$)$com\-mu\-tativity} and \textit{associativity} (including in the introduction to this paper). Note that the vertex operators do not necessarily (super)commute nor associate in the traditional sense (they do weakly (super)commute and weakly associate though; see \cite{FHL,LL}).
\end{Remark}

\begin{Remark}
With the presence of the other axioms, the duality axiom is equivalent to the Jacobi identity (see \cite{LL}), which states that for $\Z_2$-homogeneous elements $u, v\in V$,
 \begin{align*}
& x_0\iv\delta\left(\frac{x_1-x_2}{x_0}\right) Y(u,x_1)Y(v,x_2)-(-1)^{|u||v|}x_0\iv\delta\left(\frac{-x_2+x_1}{x_0}\right) Y(v,x_2)Y(u,x_1)\notag\\
&\qquad = x_2\iv\delta\left(\frac{x_1-x_0}{x_2}\right) Y(Y(u,x_0)v,x_2), 
 \end{align*}
 where $\delta(x)=\sum_{n\in \Z}x^n$. Applying the change of variables $(u,v;x_0,x_1,x_2)\leftrightarrow(v,u;-x_0,x_2,x_1)$ to the Jacobi identity and using the formal Taylor theorem, one can obtain the super version of \textit{skew-symmetry}
 \[Y(u,x)v=(-1)^{|u||v|}{\rm e}^{x\Db_V}Y(v,-x)u.\]
\end{Remark}

\begin{Definition}
 Let $V$ be a grading-restricted vertex superalgebra, and let $I$ be a linear subspace of $V$ compatible with the $\Z_2$ grading. We say $I$ is an \textit{ideal} of $V$ if $I$ is closed under $\db_V$, and for all $v\in V$ and for all $w\in I$,
 \begin{align*}
 Y(w,z)v\in I(\!(z)\!),\qquad Y(v,z)w\in I(\!(z)\!).
 \end{align*}
 In other words, for all $m\in \Z$, $v_m w\in I$ and $w_m v\in I$. \label{IdealDef}
\end{Definition}

Note that an ideal $I$ is invariant under $\Db_V$. It is also a graded ideal, i.e., if $v\in I$, then the homogeneous components of $v$ are in $I$. Therefore, the quotient $V/I$ naturally carries the structure of a grading-restricted vertex superalgebra.

\begin{Definition}
 Let $V$ be a grading-restricted vertex superalgebra. A {\it grading-restricted $V$\nobreakdash-su\-per\-module} is a $\C\times\Z_2$-graded vector space
 \[W=\coprod_{n\in\C,\alpha\in\Z_2}W^\alpha_{(n)}=\coprod_{n\in\C} W_{(n)}^0\oplus \coprod_{n\in \C} W_{(n)}^1=W^0\oplus W^1,\]
 equipped with a vertex operator map
 \begin{align*}
 Y_W\colon \ V &\to \Hom_\C{(W,W(\!(x)\!))},\qquad
 u \mapsto Y_W(u,x) = \sum_{n\in\Z}(Y_W)_n(u)x^{-n-1},
 \end{align*}
 and linear operators $\db_W$ and $\Db_W$ on $W$, with $\Db_W$ even, satisfying the following axioms:
 \begin{enumerate}\itemsep=0pt
 \item[1.] {\it $\C$-grading conditions:} Let $W_{(n)} = W_{(n)}^0 \oplus W_{(n)}^1$ for each $n\in \C$. Then for $n\in\C$, $\dim{W_{(n)}}<\infty$, and $W_{(n)} = 0$ if the real part of $n$ is sufficiently negative. We require that $\d_W w=nw$ for $w\in W_{(n)}$, and for every $v\in V$, the following \textit{$\db$-commutator formula} holds:
 \[[\db_W,Y_W(v,x)]=x\frac{{\rm d}}{{\rm d}x}Y_W(v,x)+Y_W(\db_Vv,x).\]
 \item[2.] {\it $\Z_2$-grading condition:} For $i,j\in \Z_2$, $v\in V^i$, $w \in W^j$, $Y_W(v, x)w \in W^{i+j}[[x, x^{-1}]]$. We likewise define \textit{$\Z_2$-homogeneous} and \textit{parity}, and use the notation $|v| = i, |w| = j$.
 \item[3.] {\it Identity property:} For $w\in W$, $Y_W(\vac, x)w = w$.
 \item[4.] {\it $\Db$-derivative property and $\Db$-commutator formula:} For $u\in V$,
 \[\frac{{\rm d}}{{\rm d} x}Y_W(u,x)=Y_W(\Db_V u,x)=[\Db_W,Y_W(u,x)].\]
 \item[5.] \textit{$\Db$-grading property:} $[\db_W,\Db_W]=\Db_W$.
 \item[6.] {\it Duality:} For $\Z_2$-homogeneous $u_1, u_2\in V$, $w\in W$, and $w'\in W' = \coprod_{n\in \C} W_{(n)}^*$ (the restricted dual of $W$), the following three series:
 \begin{gather*}
 \eval[w]{Y_W(u_1,z_1)Y_W(u_2,z_2)},\qquad
 (-1)^{|u_1||u_2|}\eval[w]{Y_W(u_2,z_2)Y_W(u_1,z_1)}\\
 \eval[w]{Y_W(Y(u_1,z_1-z_2)u_2,z_2)}
 \end{gather*}
 converge absolutely in the regions $|z_1|>|z_2|>0$, $|z_2|>|z_1|>0$, $|z_2|>|z_1-z_2|>0$, respectively, to a common rational function in $z_1$ and $z_2$ with the only possible poles at $z_1=0$, $z_2=0$ and $z_1=z_2$.
 \end{enumerate}
 We denote a grading-restricted $V$-supermodule by $(W, Y_W, \db_W, \Db_W)$, or just $W$. In this paper, we shall also use the abbreviation {\it $V$-supermodule} since we will not consider more general supermodules. Note that $V$ is a supermodule over itself.
\end{Definition}

Homomorphisms of supermodules are defined analogously to vertex superalgebras: they are parity-preserving maps that commute with $Y$, $\db$, and $\Db$.

\begin{Remark}
 We can view the vertex operator map $Y_W$ in a $V$-supermodule $W$ as a left-action of $V$ on $W$. Motivated by skew-symmetry, we define the linear map $Y^W_{WV}$ analogous to a~right-action:
 \begin{align*}
 Y^W_{WV} \colon \ W\otimes V \to W\bigl[\bigl[x,x\iv\bigr]\bigr],\qquad
 w\otimes v \mapsto Y^W_{WV}(w,x)v := (-1)^{|v||w|}{\rm e}^{x\Db_W}Y_W(v,-x)w
 \end{align*}
for $w\in W$, $v\in V$, $w$, $v$ both $\Z_2$-homogeneous. With this right action, a $V$-supermodule is automatically a bimodule if we view $V$ as a $\frac{1}{2}\Z$-graded meromorphic open-string vertex algebra (cf.\ \cite{FQ-Fermion-1,Q-Mod,Q-Coh,Q-Ext-1,Q-Fermion-2}).
\end{Remark}

\begin{Remark}With an argument similar to those in \cite{LL}, it is clear that the duality axiom is equivalent to the super-Jacobi identity
 \begin{align}
 &x_0^{-1}\delta\left(\frac{x_1-x_2}{x_0}\right)Y_W(u, x_1) Y_W(v, x_2) - (-1)^{|u|\cdot |v|}x_0^{-1}\delta\left(\frac{-x_2+x_1}{x_0}\right)Y_W(v, x_2)Y_W(u, x_1) \nonumber\\
&\qquad = x_2^{-1}\delta\left(\frac{x_1-x_0}{x_2}\right)Y_W(Y(u,x_0)v, x_2) \label{super-Jacobi}
 \end{align}
 for $u, v\in V$ that are $\mathbb{Z}_2$-homogeneous.
\end{Remark}

\subsection[W-valued rational functions]{$\boldsymbol{\overline{W}}$-valued rational functions}

Recall the algebraic completion $\overline{W}$ of a $V$-supermodule $W$ is precisely the full dual of the restricted dual. More precisely, if $W = \coprod_{n\in \C} W_{(n)}$, then \[\overline{W} = (W')^* = \prod_{n\in \C} W_{(n)}. \]
\begin{Definition}
 Consider the configuration space
 \[F_n\C = \{(z_1, \dots , z_n)\in \C^n \mid z_i \neq z_j,\, 1\leq i < j \leq n\}.\]
 A {\it $\overline{W}$-valued rational function} in $z_1,\dots,z_n$ with the only possible poles at $z_i=z_j$, $i\neq j$, is a~map
 \begin{align*}
 f\colon \ F_n\C \to \overline{W},\qquad
 (z_1,\dots,z_n) \mapsto f(z_1,\dots,z_n)
 \end{align*}
 such that for any $w'\in W'$,
$\lrgl{w',f(z_1,\dots,z_n)}$
 is a rational function in $z_1\dots,z_n$ with the only possible poles at $z_i=z_j$, $i\neq j$. We use $\widetilde{W}_{z_1,\dots,z_n}$ to denote the space of such functions.
\end{Definition}

\begin{Example}
 Let $u_1, \dots , u_n\in V$, $w'\in W'$ and $w\in W$ be a vacuum-like vector in $W$, i.e., $Y_W(v, x)w\in W[[x]]$. In this case, we know that the series
 \begin{align}
 \langle w', Y_W(u_1, z_1) \cdots Y_W(u_n, z_n)w\rangle \label{series-product}
 \end{align}
 converges absolutely to a rational function with the only possible poles at $z_i = z_j$, $1\leq i < j \leq n$, denoted by
 \begin{align}
 R (\langle w', Y_W(u_1, z_1)\cdots Y_W(u_n, z_n) w\rangle ); \label{R-product}
 \end{align}
 (\ref{R-product}) is the analytic continuation of the series (\ref{series-product}) and thus makes sense whenever $z_i \neq z_j$, $1\leq i < j \leq n$. The linear map
$w'\mapsto R (\langle w', Y_W(u_1, z_1)\cdots Y_W(u_n, z_n) w\rangle )$
 defines an element in $\overline{W}$ for every $(z_1, \dots , z_n)\in F_n\C$. If we denote this element in $\overline{W}$ by
 \begin{align}
 E (Y_W(u_1, z_1)\cdots Y_W(u_n, z_n)w ),\label{E-product}
 \end{align}
 then the map
$(z_1, \dots , z_n)\mapsto E (Y_W(u_1, z_1)\cdots Y_W(u_n, z_n)w )$
 is a $\overline{W}$-valued rational function. We shall omit the reference of $(z_1, \dots , z_n)\in F_n\C$ and directly say that~(\ref{E-product}) is a~$\overline{W}$-valued rational function.
\end{Example}
\begin{Notation}
 Generally, for a series
 \begin{align*}
 \sum_{m_1, \dots , m_n\in \Z} w_{m_1\dots m_n}z_1^{-m_1-1}\cdots z_n^{-m_n-1}
\end{align*}
in $W\bigl[\bigl[z_1, z_1^{-1}, \dots , z_n, z_n^{-1}\bigl]\bigr]$, such that for every $w'\in W'$, the complex series
\begin{align}
\biggl\langle w', \sum_{m_1, \dots , m_n\in \Z} w_{m_1\dots m_n} z_1^{-m_1-1}\cdots z_n^{-m_n-1}\biggr\rangle\label{V-series-pair-v'}
\end{align}
converges absolutely in a certain region to a rational function with the only possible poles at $z_i = z_j$, $1\leq i < j \leq n$, we may use the notation
 \begin{align}
 R\biggl(\biggl\langle w', \sum_{m_1, \dots , m_n\in \Z} w_{m_1\dots m_n} z_1^{-m_1-1}\cdots z_n^{-m_n-1}\biggr\rangle\biggr) \label{V-series-pair-v'-sum}
 \end{align}
 for the sum of the complex series (\ref{V-series-pair-v'}), which makes sense everywhere on $F_n\C$. We will use the notation
 \begin{align}
 E\biggl(\sum_{m_1, \dots , m_n\in \Z} w_{m_1\dots m_n} z_1^{-m_1-1}\cdots z_n^{-m_n-1}\biggr)\label{E-of-V-series}
 \end{align}
 to denote the corresponding $\overline{W}$-valued rational function, which may be viewed as the analytic continuation of the sum of the series $f(z_1, \dots , z_n)$ with coefficients in $W$.
\end{Notation}
\begin{Remark}
Note that (\ref{V-series-pair-v'-sum}) and (\ref{E-of-V-series}) make sense everywhere on $F_n\C$, while~(\ref{V-series-pair-v'}) may converge in a much smaller region. For example, let $u_1, u_2\in V$ be $\Z_2$-homogeneous, $w'\in W'$, and $w\in W$ be a vacuum-like vector. Then \[\langle w', Y(u_1, z_1)Y(u_2, z_2)w\rangle \]
 converges absolutely in the region $|z_1|>|z_2|$; the series
 \[(-1)^{|u_1||u_2|}\langle w', Y(u_2, z_2)Y(u_1, z_1)w\rangle \]
 converges absolutely in the region $|z_2|>|z_1|$. Since the regions of convergence of these two series do not intersect, we cannot equate them. However, we may say that
 \[R (\langle w', Y(u_1, z_1)Y(u_2, z_2)w\rangle ) = (-1)^{|u_1||u_2|} R (\langle w', Y(u_2, z_2)Y(u_1, z_1)w\rangle ),\]
 because these two series converge to a common rational function. We may also say that
 \[E (Y(u_1, z_1)Y(u_2, z_2)w ) = (-1)^{|u_1||u_2|} E (Y(u_2, z_2)Y(u_1, z_1)w ),\]
 since these two series converge to the same $\overline{W}$-valued rational functions. The use of $\overline{W}$-valued rational functions is crucial for the construction of the cohomology theory of vertex algebras and superalgebras.
\end{Remark}

\begin{Notation}For convenience, we introduce the following notations:
 \begin{align}
 E_{W}^{(n)}(v_1\otimes \cdots \otimes v_n;w;z_1, \dots , z_{n}) = E (Y_W(v_1, z_1) \cdots Y_W(v_n, z_n)w ); \label{EW-n}
 \end{align}
(\ref{EW-n}) defines $\overline{W}$-valued rational functions, but not in $\widetilde{W}_{z_1, \dots , z_{n}}$ unless $w$ is vacuum-like. In case $W = V$, we have
 \begin{align*}
 E_{V}^{(n)}(v_1 \otimes \cdots \otimes v_n; \vac; z_1, \dots , z_n) = E (Y(v_1, z_1)\cdots Y(v_n, z_n)\vac )
 \end{align*}
 which defines a $\overline{V}$-valued rational function in $\widetilde{V}_{z_1, \dots , z_n}$.
\end{Notation}

\begin{Remark}A linear map $f$ with domain $W$ can be naturally extended to $\overline{W}$, provided that a~certain convergence condition holds. More precisely, given an element $\bar w \in \overline{W}$, we first apply the projection $P_n\colon \overline{W}\to W_{(n)}$ to obtain an element $P_n \bar{w}\in W$. Then $f(P_n\bar{w})$ makes sense. The extension is well defined if the series $\sum_{n\in \C} f(P_n \bar{w})$ converges to a well-defined element in $\overline{W}$, i.e., for each $r\in \C$, only finitely many $n$ makes $P_r f(P_n\bar{w})$ nonzero.
\end{Remark}

\begin{Remark}The product of vertex operators may also be understood via this procedure. For example, fix $u_1, u_2\in V$, $w\in W$. Then for a fixed $z_2\in \C^\times$, $Y_W(u_2, z_2)w\in \overline{W}$. To act $Y_W(u_1,z_1)$ on this element in $\overline{W}$, we follow the same procedure and set
 \[Y_W(u_1, z_1)Y_W(u_2, z_2)w = \sum_{n\in \C} Y_W(u_1, z_1) P_n Y_W(u_2, z_2)w.\]
 From the duality axiom, the series converges absolutely if $|z_1|>|z_2|>0$ defines an element in~$\overline{W}$. By an analytic continuation process, we obtain a $\overline{W}$-valued rational function, which is an element in $\overline{W}$ that makes sense everywhere nonzero on~$F_n\C$.
\end{Remark}

\subsection[Cochain construction I: D-derivative property and d-conjugation property]{Cochain construction I: $\boldsymbol{\Db}$-derivative property\\ and $\boldsymbol{\db}$-conjugation property}

We shall build a cochain out of linear maps $\Phi\colon V^{\otimes n}\to\widetilde{W}_{z_1,\dots,z_n}$. For convenience, we modify the notation as
$\Phi(v_1\otimes \cdots \otimes v_n; z_1, \dots , z_n)$
instead of the notation in \cite{H-Coh}, which is more conceptual but involves too many parentheses. In order to maintain the geometric interpretation, we would like them to be compatible with the operators $\db$ and $\Db$ on $V$ and $W$.
\begin{Definition}
 For $n\in\Z_+$, a linear map $\Phi\colon V^{\otimes n}\to\widetilde{W}_{z_1,\dots,z_n}$ is said to have the {\it $\Db$-derivative property} if
 \begin{align*}
 & \frac{\pd}{\pd z_i}\lrgl{w',\Phi(v_1\otimes\cdots\otimes v_n;z_1,\dots,z_n)}\\
&\qquad = \lrgl{w',\Phi(v_1\otimes\cdots\otimes v_{i-1}\otimes\Db_V v_i\otimes v_{i+1}\otimes\cdots\otimes v_n; z_1,\dots,z_n)}
 \end{align*}
 for $i=1,\dots,n$, $v_1,\dots,v_n\in V$, and $w'\in W'$; and
 \begin{align*}
 & \left(\frac{\pd}{\pd z_1}+\cdots+\frac{\pd}{\pd z_n}\right)\lrgl{w',\Phi(v_1\otimes\cdots\otimes v_n;z_1,\dots,z_n)}\\
& \qquad= \lrgl{w',\Db_W \Phi(v_1\otimes\cdots\otimes v_n;z_1,\dots,z_n)}.
 \end{align*}
\end{Definition}
\begin{Definition}
 For $n\in \Z_+$, a linear map $\Phi\colon V^{\otimes n}\to\widetilde{W}_{z_1,\dots,z_n}$ is said to have the {\it $\db$-conjugation property} if for $v_1,\dots,v_n\in V$, $w\in W$, $w'\in W'$, $(z_1,\dots,z_n)\in F_n\C$, and $z\in \C^\times$ such that $(zz_1,\dots,zz_n)\in F_n\C$,
 \begin{align*}
 \bigl\langle w',z^{\db_W}\Phi(v_1\otimes\cdots\otimes v_n;z_1,\dots,z_n)\bigr\rangle
 = \bigl\langle w',\Phi\bigl(z^{\db_V}v_1\otimes\cdots\otimes z^{\db_V}v_n; zz_1,\dots,zz_n\bigr)\bigr\rangle .
 \end{align*}
Note that despite the fact that $V$ and $W$ admits half-integral grading, since $\Phi$ is assumed to be a rational function, the half-integer power introduced should be balanced on both sides. So it is not necessary to specify a branch of~$z^{1/2}$ for this identity.
\end{Definition}

The following proposition follows directly from the $\Db$-derivative property and the formal Taylor theorem. For a treatment of the formal Taylor theorem, see~\cite{LL}.
\begin{Proposition}
 Let $\Phi\colon V^{\otimes n}\to\widetilde{W}_{z_1,\dots,z_n}$ be a linear map with the $\Db$-derivative property. Then for $v_1,\dots,v_n\in V$, $w'\in W'$, $(z_1,\dots,z_n)\in F_n\C$, $z\in \C$, such that $(z_1+z,\dots,z_n+z)\in F_n\C$,
 \begin{align*}
 \bigl\langle w' ,{\rm e}^{z\Db_W} \Phi(v_1\otimes\cdots\otimes v_n;z_1,\dots,z_n)\bigr\rangle
 =\bigl \langle w',\Phi(v_1\otimes\cdots\otimes v_n;z_1+z,\dots,z_n+z)\bigr\rangle ,
 \end{align*}
 and for $(z_1,\dots,z_n)\in F_n\C$, $z\in\C$, such that $(z_1,\dots,z_{i-1},z_i+z,z_{i+1},\dots,z_n)\in F_n\C$, the power series of
 \begin{equation}
 \bigl \langle w',\Phi\bigl(v_1\otimes\cdots\otimes v_{i-1}\otimes {\rm e}^{z\Db_V}v_{i}\otimes v_{i+1}\otimes\cdots\otimes v_n;z_1,\dots,z_n\bigr)\bigr\rangle \label{Taylor1}
 \end{equation}
 in $z$ is equal to the power series expansion of
 \begin{equation}
\bigl \langle w',\Phi(v_1\otimes\cdots\otimes v_n;z_1,\dots,z_{i-1},z_i+z,z_{i+1},\dots,z_n)\bigr\rangle \label{Taylor2}
 \end{equation}
 in $z$. In particular, the power series \eqref{Taylor1} in $z$ is absolutely convergent to \eqref{Taylor2} in the disk $|z|<\min_{i\neq j}\{|z_i-z_j|\}$.
\end{Proposition}

\subsection{Cochain construction II: composable condition} To make sense of the coboundary operator, we need to impose the following condition on the linear maps $V^{\otimes n}\to \widetilde{W}_{z_1,\dots ,z_n}$.
\begin{Definition}\label{m-composable}
 Let $\Phi\colon V^{\otimes n}\to \widetilde{W}_{z_1,\dots,z_n}$ be a linear map. For $m\in \N$, $\Phi$ is said to be {\it composable with $m$ vertex operators} if the following condition is satisfied:
 Fix arbitrary $l_0\in\N$, $l_1,\dots,l_n\in\Z_+$ such that $l_0+l_1+\dots+l_n=m+n$, $v_1,\dots,v_{m+n}\in V$ and $w'\in W'$. Set
 \begin{align*}
 \Psi_i&= E_V^{(l_i)}(v_{l_0+\cdots + l_{i-1} +1} \otimes \cdots \otimes v_{l_0+\cdots + l_{i-1}+l_i};\vac; z_{l_0+\cdots +l_{i-1} + 1}-\zeta_i, \dots , z_{l_0+\cdots +l_{i-1} + l_i}-\zeta_i)\\
 & =E\bigl(Y(v_{l_0+\cdots+l_{i-1}+1},z_{l_0+\cdots+l_{i-1}+1}-\zeta_i)\cdots Y(v_{l_0+\cdots+l_{i-1}+l_i},z_{l_0+\cdots+l_{i-1}+l_i}-\zeta_i)\vac\bigr)
 \end{align*}
 for $i=1,\dots,n$. Then the series
 \begin{align*}
 \sum_{r_0\in\C,r_1,\dots,r_n\in\frac{1}{2}\Z}\lrgl{w',E(Y_W(v_1,z_1)\cdots Y_W(v_{l_0},z_{l_0}) P_{r_0}(\Phi(P_{r_1}\Psi_1\otimes\cdots\otimes P_{r_n}\Psi_n;\zeta_1,\dots,\zeta_n)))}
 \end{align*}
 is absolutely convergent when
 \begin{align*}
 & z_i\neq z_j, \qquad 1\leq i < j \leq l_0,\\
 & |z_i|>|\zeta_j|+|z_k-\zeta_j|, \qquad 1\leq i\leq l_0,\qquad 1\leq j\leq n,\qquad \sum_{\beta=0}^{j-1}l_\beta+1\leq k \leq \sum_{\beta=0}^{j-1}l_\beta + l_j, \\
 & |z_{l_0+\cdots+l_{i-1}+p}-\zeta_i|+|z_{l_0+\cdots+l_{j-1}+q}-\zeta_j|<||\zeta_i|-|\zeta_j||, \\ &1\leq i < j \leq n,\qquad 1\leq p\leq l_i,\qquad 1\leq q \leq l_j.
 \end{align*}
 Moreover, there exist positive integers $N(v_i,v_j)$ depending only on $v_i$ and $v_j$ for $i,j=1,\dots,\allowbreak {m+n}$, $i\neq j$ such that the sum can be analytically extended to a rational function in $z_1,\dots,z_{m+n}$ independent of $\zeta_1,\dots,\zeta_n$, with the only possible poles at $z_i=z_j$ of order less than or equal to $N(v_i,v_j)$ for $1\leq i<j\leq m+n$.
\end{Definition}

\begin{Notation}
 For $n\in\Z_+$, let $\widehat{C}^n_0(V,W)$ denote the subspace of $\Hom_\C\bigl(V^{\otimes n},\widetilde{W}_{z_1,\dots,z_n}\bigr)$ satisfying both the $\Db$-derivative and $\db$-conjugation properties, let $\widehat{C}^n_m(V,W)$ denote the subspace of $\widehat{C}^n_0(V,W)$ that is further composable with $m$ vertex operators, and let $\widehat{C}^0_m(V,W)=W$. Then we have
$\widehat{C}^n_m(V,W)\sseq \widehat{C}^n_{m-1}(V,W)$
for $m\in\Z_+$. Additionally let
$\widehat{C}^n_\infty(V,W)=\bigcap_{m\in\N}\widehat{C}^n_m(V,W)$.
\end{Notation}

\begin{Example}\label{vertexExam}
 For weight 0 $w\in W$ satisfying $Y_W(v,x)w\in W[[x]]$ and $\Db_W(w)=0$, the $\overline{W}$-valued rational function \smash{$E^{(n)}_{W,w}(v_1\otimes\cdots\otimes v_n)$} for $v_1,\dots,v_n\in V$ gives a linear map
 \begin{align*}
 E^{(n)}_{W;w}\colon \ V^{\otimes n} \to \widetilde{W}_{z_1,\dots,z_n},\qquad
 v_1\otimes\cdots\otimes v_n \mapsto E^{(n)}_W(v_1\otimes\cdots\otimes v_n;w;z_1, \dots , z_n).
 \end{align*}
 By the properties of $Y_W$-operators, this map has the $\Db$-derivative and $\db$-conjugation properties. Since products and iterates of vertex operators converge, \smash{$E_{W, w}^{(n)}$} is composable with any number of vertex operators. In particular, \smash{$E^{(n)}_{V}$} is such a map. So the spaces \smash{$\widehat{C}_{\infty}^n(V, V)$} and \smash{$\widehat{C}_{m}^n(V, V)$} are nonempty for every $m, n\in \Z_+$.
\end{Example}

\begin{Notation} \label{vertexNotation}
 For convenience, we define the linear maps
 \begin{align*}
& \Phi\circ\bigl(E^{(l_1)}_V\otimes\cdots\otimes E^{(l_n)}_V\bigr)\colon V^{\otimes(m+n)}\to \widetilde{W}_{z_1,\dots,z_{m+n}},\\
 &E^{(m)}_W\circ_{m+1}\Phi\colon \ V^{\otimes(m+n)}\to \widetilde{W}_{z_1,\dots,z_{m+n}},\\
 &E^{W;(m)}_{WV}\circ_{0}\Phi\colon \ V^{\otimes(m+n)}\to \widetilde{W}_{z_1,\dots,z_{m+n}}
\end{align*}
by
\begin{align*}
 & \bigl(\Phi\circ\bigl(E^{(l_1)}_V\otimes\cdots\otimes E^{(l_n)}_V\bigr)\bigr)(v_1\otimes\cdots\otimes v_{m+n}; z_1, \dots , z_{m+n})\\
& \qquad = E\biggl(\sum_{r_1,\dots , r_n\in \frac{1}{2}\Z}\Phi(P_{r_1} E^{(l_1)}_V(v_1\otimes\cdots\otimes v_{l_1};\vac;z_1-\zeta_1, \dots , z_{l_1}-\zeta_1)\\
 & \qquad\quad{} \otimes\cdots\otimes P_{r_n}E^{(l_n)}_V(v_{l_1+\cdots+l_{n-1}+1}\otimes\cdots\otimes v_{l_1+\cdots+l_n};\vac;\\
 &\qquad \qquad \quad z_{l_1+\cdots+l_{n-1}+1}-\zeta_n, \dots, z_{l_1+\cdots+l_n}-\zeta_n); \zeta_1, \dots, \zeta_n)\biggr),\\
 & \bigl(E^{(m)}_W\circ_{m+1}\Phi\bigr)(v_1\otimes\cdots\otimes v_{m+n};z_1, \dots , z_{m+n})\\
&\qquad = E\biggl(\sum_{r\in \C}E^{(m)}_W(v_1\otimes\cdots\otimes v_m;P_r\Phi(v_{m+1}\otimes\cdots\otimes v_{m+n};z_{m+1}, \dots , z_{m+n});z_1, \dots , z_m)\!\biggr),\\
 & \bigl(E^{W;(m)}_{WV}\circ_{0}\Phi\bigr)(v_1\otimes\cdots\otimes v_{m+n};z_1, \dots , z_{m+n})\\
& \qquad = E\biggl (\sum_{r\in \C} \sum_{i = 0}^1(-1)^{\sum\limits_{j=1}^m i|v_{n+j}|} E^{(m)}_W(v_{n+1}\otimes \cdots \otimes v_{n+m}; P_{r,i} \Phi(v_1\otimes \cdots \otimes v_n; \\
 &\qquad\qquad\quad z_1, \dots , z_n); z_{n+1}, \dots , z_{n+m})\biggr).
\end{align*}
respectively, where for $\overline{w}\in \overline{W}$, $P_r\overline{w} = P_{r,0}\overline{w} + P_{r,1}\overline{w}$ with $P_{r,i}\overline{w}\in W^i$, $i = 0, 1$.

In addition, if $l_1=\cdots=l_{i-1}=l_{i+1}=\cdots=l_n=1$ and $l_i=k$ for some $i$, $1\leq i\leq n$ and $k\in\Z_+$, we use the notation $\Phi\circ_iE^{(k)}_V$ for the first expression.
\end{Notation}

To describe the square-zero extensions and first-order deformations, we consider the following $1/2$-composable condition.

\begin{Definition}
We say $\Phi\in\widehat{C}^2_0(V,W)$ satisfies the $1/2$-composable condition, if for $v_1,v_2,\allowbreak v_3\in V$, and $w'\in W'$, the series
\begin{align*}
 &\sum_{r\in\C}\bigl(\bigl\langle w',E^{(1)}_W(v_1;P_r(\Phi(v_2\otimes v_3;z_2-\zeta,z_3-\zeta));z_1,\zeta)\bigr\rangle \\
 &\qquad + \bigl\langle w',\Phi\bigl(v_1\otimes P_r\bigl(E^{(2)}_V(v_2\otimes v_3;\vac;z_2-\zeta,z_3-\zeta)\bigr);z_1,\zeta\bigr)\bigr\rangle \bigr)
\end{align*}
and
\begin{align*}
 &\sum_{r\in\C}\bigl(\bigl\langle w',\Phi\bigl(P_r\bigl(E^{(2)}_V(v_1\otimes v_2;\vac;z_1-\zeta,z_2-\zeta)\bigr)\otimes v_3;\zeta,z_3\bigr)\bigr\rangle \\
 &\qquad{} +\bigl\langle w',E^{W,(1)}_{WV}(P_r(\Phi(v_1\otimes v_2;z_1-\zeta,z_2-\zeta));v_3)(\zeta;z_3)\bigr\rangle \bigr)
\end{align*}
are absolutely convergent in the regions \[|z_1-\zeta|>|z_2-\zeta|,\qquad |z_3-\zeta|>0,\qquad \text{and}\qquad |\zeta-z_3|>|z_1-\zeta|,\qquad |z_2-\zeta|>0,\]
respectively, and can be analytically extended to rational functions in $z_1$, $z_2$, and $z_3$ with the only possible poles at $z_i = z_j$, $1\leq i < j \leq 3$.
Let \smash{$\widehat{C}^2_{1/2}{(V,W)}$} be the subspace of $1/2$-composable linear maps in \smash{$\widehat{C}_0^2(V, W)$}. Clearly, \smash{$\widehat{C}^2_m(V,W)\sseq\widehat{C}^2_{1/2}(V,W)$} for every $m\in\Z_+$.
\end{Definition}

\subsection{Cochain construction III: shuffle condition} For $V$-supermodules, we will further restrict to particular linear maps that are parity-preserving and have zero super shuffle sums.

\begin{Definition}
 For $n\in \Z_+$, a linear map $\Phi\colon V^{\otimes n} \to \widetilde{W}_{z_1,\dots , z_n}$ is \textit{parity-preserving} if for every $v_1, \dots , v_n \in V$ that are $\Z_2$-homogeneous, and $(z_1, \dots , z_n)\in F_n\C$,
 \[\Phi(v_1\otimes \cdots \otimes v_n; z_1, \dots , z_n)\in \overline{W^{|v_1|+\cdots + |v_n|}}.\]
 Equivalently, for every $r\in \C$,
 \[P_r \Phi(v_1\otimes \cdots \otimes v_n; z_1, \dots , z_n) \in W^{|v_1|+\cdots + |v_n|}_{(r)}.\]
\end{Definition}

\begin{Definition}
Let $f\in\widetilde{W}_{z_1,\dots,z_n}$, and let $\sigma\in\Sym_n$. We define a left action of $\Sym_n$ on $\widetilde{W}_{z_1,\dots,z_n}$ by
\[(\sigma(f))(z_1,\dots,z_n)=f(z_{\sigma(1)},\dots,z_{\sigma{(n)}}).\]
Now let $\Phi\in\Hom_\C\bigl(V^{\otimes n},\widetilde{W}_{z_1,\dots,z_n}\bigr)$, and let $v_1,\dots,v_n\in V$. We extend the action of $\Sym_n$ to $\Hom_\C\bigl(V^{\otimes n},\widetilde{W}_{z_1,\dots,z_n}\bigr)$ by
\[(\sigma(\Phi))(v_1\otimes\cdots\otimes v_n)=\sigma(\Phi(v_{\sigma(1)}\otimes\cdots\otimes v_{\sigma(n)})).\]
More precisely,
\begin{align*}
 (\sigma(\Phi))(v_1\otimes\cdots \otimes v_n;z_1, \dots , z_n)
 = \Phi(v_{\sigma(1)} \otimes \cdots \otimes v_{\sigma(n)};z_{\sigma(1)}, \dots , z_{\sigma(n)}).
\end{align*}
We shall use the notation
\[\begin{pmatrix}
 1 & 2 & \cdots & n-1 & n\\
 i_1 & i_2 & \cdots & i_{n-1} & i_n
\end{pmatrix}\]
for the permutation sending $j$ to $i_j$ for $j=1,\dots,n$.
\end{Definition}

\begin{Definition}
 Recall the set of shuffles $J_p(1,\dots,n)$ from Section~\ref{section2}. Define the set $J_{n;p}\iv=\bigl\{\sigma\iv\mid \sigma\in J_p(1,\dots,n)\bigr\}$. For $m,n\in\N$ or $m=\frac{1}{2}$, $n=2$, we define the subspace $C^n_m(V,W)\sseq\widehat{C}^n_m(V,W)$ consisting of parity-preserving linear maps \smash{$\Phi\colon V^{\otimes n}\to \widetilde{W}_{z_1, \dots , z_n}$}, such that
 \[\sum_{\sigma\in J\iv_{n;p}}(-1)^\sigma(-1)^{\sigma^o}\sigma(\Phi(v_{\sigma(1)}\otimes\cdots\otimes v_{\sigma(n)}))=0\]
 for $1\leq p\leq n-1$.
 We also define
 \[C^n_\infty(V,W)=\bigcap_{m\in\N}C^n_m(V,W).\]
\end{Definition}

\subsection{Coface maps}
To show that $\widehat{C}^n_m(V,W)$ and $C^n_m(V, W)$ form a cochain with the appropriate coboundary map, we apply simplicial methods. We define the following coface maps analogously to the Hochschild maps in \cite[Chapter 9]{Weibel}.
\begin{Definition}
 For $i=0,\dots,n+1$, we define the coface maps
\[\hat{\pd}^{n;i}_m\colon \ \widehat{C}^n_m(V,W)\to \widehat{C}^{n+1}_{m-1}(V,W)\]
by
 \[\hat{\pd}^{n;i}_m(\Phi) = \begin{cases}
 E^{(1)}_W\circ_2\Phi & \text{if }i=0,\\
 \Phi\circ_i E^{(2)}_V & \text{if }1\leq i\leq n,\\
 E^{W,(1)}_{WV}\circ_0\Phi & \text{if }i=n+1
\end{cases}\]
(using Notation~\ref{vertexNotation}).
\end{Definition}
To check that the maps so defined form a family of coface maps, it suffices to prove the following technical lemma.
\begin{Lemma}\label{CosimpLemma}
 For $n\in\Z_{\geq0}$, $i,j=0,\dots,n+1$, $i<j$, and $m\in \Z_++1$,
 \[\hat{\pd}^{n;j}_{m-1}\circ \hat{\pd}^{{n-1};i}_m=\hat{\pd}^{n;i}_{m-1}\circ\hat{\pd}^{{n-1};j-1}_m.\]
\end{Lemma}
\begin{proof}
 Clearly, the image of each $\hat{\pd}^{n;i}_m$ is indeed composable with $m-1$ vertex operators. Because of absolute convergence, we may change the order of iterated sums. Now let $\Phi\in\widehat{C}^{n}_m(V,W)$.
 Then for $v_1,\dots,v_{n+1}\in V$ that are $\Z_2$-homogeneous, $w'\in W'$, and $(z_1,\dots,z_{n+1})\in F_{n+1}\C$, we can realize the coface maps explicitly:
\begin{align*}
& \bigl\langle w',\bigl(\hat{\pd}^{n;i}_m(\Phi)\bigr)(v_1\otimes\cdots\otimes v_{n+1};z_1,\dots,z_{n+1})\bigr\rangle \\
 &\qquad = \begin{cases}
 R\bigl(\bigl\langle w',Y_W(v_1,z_1)\Phi(v_2\otimes\cdots\otimes v_{n+1};z_2,\dots,z_{n+1})\bigr\rangle \bigr) & \text{if }i=0,\\
 R\bigl(\bigl\langle w',\Phi(v_1\otimes\cdots\otimes v_{i-1}\otimes Y(v_i,z_i-\zeta_i)Y(v_{i+1},z_{i+1}-\zeta_i)\vac
 \\ \qquad \otimes v_{i+2}\otimes\cdots\otimes v_{n+1};z_1,\dots,z_{i-1},\zeta_i,z_{i+2},\dots,z_{n+1})\bigr\rangle \bigr) & \text{if }1\leq i\leq n,\\
 R\bigl(\bigl\langle w',Y_W(v_{n+1},z_{n+1})P_{(0)}(\Phi(v_1\otimes\cdots\otimes v_n;z_1,\dots,z_n))\bigr\rangle \bigr)\\
 \qquad{}+(-1)^{|v_{n+1}|}R\bigl(\bigl\langle w',Y_W(v_{n+1},z_{n+1})\\
 \hspace{100pt} P_{(1)}(\Phi(v_1\otimes\cdots\otimes v_n;z_1,\dots,z_n))\bigr\rangle\bigr) & \text{if }i=n+1,
 \end{cases}
 \end{align*}
 where $P_{(0)}$ and $P_{(1)}$ are the projections onto the even or odd parts of the relevant subspaces of $\overline{W}$-valued rational functions, respectively. Because the above is independent of $\zeta_i$, we can take $\zeta_i=z_{i+1}$, so for $1\leq i\leq n$ we get
 \begin{align*}
& \bigl\langle w',\bigl(\hat{\pd}^{n;i}_m(\Phi)\bigr)(v_1\otimes\cdots\otimes v_{n+1};z_1,\dots,z_{n+1})\bigr\rangle \\
&\qquad =R(\lrgl{w',\Phi(v_1\otimes\cdots\otimes v_{i-1}\otimes(Y(v_i,z_i-z_{i+1})v_{i+1})\otimes v_{i+2}\otimes\cdots\otimes v_{n+1};\\
 &\qquad\qquad z_1,\dots,z_{i-1},z_{i+1},\dots,z_{n+1})}).
 \end{align*}
 We now break things into cases, using lemmas in \cite{Thesis}. Let $\Phi\in\widehat{C}^{n-1}_m(V,W)$.

1. First, we show $\hat{\pd}^{n;j}_{m-1}\circ \hat{\pd}^{{n-1};0}_m=\hat{\pd}^{n;0}_{m-1}\circ \hat{\pd}^{{n-1};j-1}_m$, for $2\leq j\leq n$. We have
 \begin{gather*}
 \hat{\pd}^{n;j}_{m-1}\bigl(\hat{\pd}^{{n-1};0}_m \Phi\bigr)=\bigl(E^{(1)}_W\circ_2\Phi\bigr)\circ_j E^{(2)}_V,\\
\hat{\pd}^{n;0}_{m-1} \bigl(\hat{\pd}^{{n-1};j-1}_m\Phi\bigr)= E^{(1)}_W\circ_2\bigl(\Phi\circ_{j-1} E^{(2)}_V\bigr).
\end{gather*}
 By \cite[Lemma 4.19]{Thesis}, these are equal, proving the first case. 

2. Next we show $\hat{\pd}^{n;1}_{m-1}\circ \hat{\pd}^{{n-1};0}_m=\hat{\pd}^{n;0}_{m-1}\circ \hat{\pd}^{{n-1};0}_m$. We have{\samepage
 \begin{gather*}
 \begin{split}
& \hat{\pd}^{n;1}_{m-1}\bigl(\hat{\pd}^{{n-1};0}_m\Phi\bigr)=\bigl(E^{(1)}_W\circ_2\Phi\bigr)\circ_1 E^{(2)}_V,\\
& \hat{\pd}^{n;0}_{m-1}\bigl(\hat{\pd}^{{n-1};0}_m\Phi\bigr)=E^{(1)}_W\circ_2\bigl(E^{(1)}_W\circ_2\Phi\bigr).
\end{split}
\end{gather*}
 By \cite[Lemma~4.18]{Thesis}, these are equal, proving the second case.} 

3. Now we show $\hat{\pd}^{n;n+1}_{m-1}\circ \hat{\pd}^{{n-1};0}_m=\hat{\pd}^{n;0}_{m-1}\circ \hat{\pd}^{{n-1};n}_m$. First consider
 \[\hat{\pd}^{n;n+1}_{m-1}\bigl(\hat{\pd}^{{n-1};0}_m\Phi\bigr)=E^{W,(1)}_{WV}\circ_0\bigl(E^{(1)}_W\circ_2\Phi\bigr).\]
 Evaluating at $v_1,\dots,v_{n+1}\in V$ (all $\Z_2$-homogeneous) and $(z_1,\dots,z_{n+1})\in F_{n+1}\C$ and pairing with $w'\in W'$, we obtain
 \begin{align}
& \bigl\langle w',E\bigl(E^{W,(1)}_{WV}\bigl((E^{(1)}_W\circ_2\Phi)(v_1\otimes\cdots\otimes v_n);v_{n+1}\bigr)\bigr)(z_1,\dots,z_{n+1})\bigr\rangle\nonumber\\
 & =R\bigl(\bigl\langle w',E^{W,(1)}_{WV}\bigl(\bigl(E^{(1)}_W\circ_2\Phi\bigr)(v_1\otimes\cdots\otimes v_n);v_{n+1}\bigr)(z_1,\dots,z_{n+1})\bigr\rangle \bigr)\nonumber\\
 & =R\bigl(\bigl\langle w',E^{W,(1)}_{WV}\bigl(E^{(1)}_W(v_1;\Phi(v_2\otimes\cdots\otimes v_n));v_{n+1}\bigr)(z_1,\dots,z_n)\bigr\rangle\bigr)\label{simplex-Case3-1}\\
 & =\sum_{j=0}^1(-1)^{j|v_{n+1}|}R\bigl(\bigl\langle w',E^{(1)}_W\bigl(v_{n+1};P_{(j)}\bigl(E^{(1)}_W(v_1;\Phi(v_2\otimes\cdots\otimes v_n))\bigr)\bigr) (z_{n+1},z_1,\dots,z_n)\bigr\rangle \bigr)\nonumber\\
 & =\sum_{j=0}^1(-1)^{|v_{n+1}|(|v_1|+j)}R\bigl(\langle w',Y_W(v_{n+1},z_{n+1})Y_W(v_1,z_1)
 P_{(j)}(\Phi(v_2\otimes\cdots\otimes v_n;z_2,\dots,z_n))\rangle \bigr).\!\nonumber
 \end{align}

 On the other hand, consider
 \[(\hat{\pd}^{n;0}_{m-1})\bigl(\hat{\pd}^{{n-1};n}_m\bigr)\Phi=E^{(1)}_W\circ_2\bigl(E^{W,(1)}_{WV}\circ_0\Phi\bigr).\]
 Evaluating at $v_1,\dots,v_n$ and $(z_1,\dots,z_{n+1})$ and pairing with $w'$, we obtain
 \begin{align*}
& \bigl\langle w',E\bigl(E^{(1)}_W(v_1;\bigl(E^{W,(1)}_{WV}\circ_0\Phi\bigr)(v_2\otimes\cdots\otimes v_{n+1}))\bigr)(z_1,\dots,z_{n+1})\bigr\rangle\\
 &=R\bigl(\bigl\langle w',E^{(1)}_W\bigl(v_1;\bigl(E^{W,(1)}_{WV}\circ_0\Phi\bigr)(v_2\otimes\cdots\otimes v_{n+1})\bigr)(z_1,\dots,z_{n+1})\bigr\rangle\bigr)\\
 &=R\bigl(\bigl\langle w',E^{(1)}_W\bigl(v_1;E^{W,(1)}_{WV}(\Phi(v_2\otimes\cdots\otimes v_n);v_{n+1})\bigr)(z_1,\dots,z_{n+1})\bigr\rangle\bigr)\\
 &=\sum_{j=0}^1(-1)^{j|v_{n+1}|}R\bigl(\bigl\langle w',E^{(1)}_W\bigl(v_{1};E^{(1)}_W(v_{n+1};P_{(j)}(\Phi(v_2\otimes\cdots\otimes v_n)))\bigr)(z_1,z_{n+1},\dots,z_n)\bigr\rangle \bigr)\\
 &=\sum_{j=0}^1(-1)^{j|v_{n+1}|}R\bigl(\bigl\langle w',Y_W(v_1,z_1)Y_W(v_{n+1},z_{n+1})P_{(j)}(\Phi(v_2\otimes\cdots\otimes v_n;z_2,\dots,z_n))\bigr\rangle \bigr).
 \end{align*}
 Using the supercommutativity of the algebra, this is the same as{\samepage
 \begin{align}
 &\sum_{j=0}^1(-1)^{|v_{n+1}||v_1|}(-1)^{j|v_{n+1}|}R(\lrgl{w',Y_W(v_{n+1},z_{n+1})Y_W(v_1,z_1) \nonumber\\
 &\qquad \times P_{(j)}(\Phi(v_2\otimes\cdots\otimes v_n;z_2,\dots,z_n))}).\label{simplex-Case3-2}
 \end{align}
 We see that $(\ref{simplex-Case3-1})=(\ref{simplex-Case3-2})$, proving the third case.}

4. Next we show $\hat{\pd}^{n;j}_{m-1}\circ \hat{\pd}^{{n-1};i}_m=\hat{\pd}^{n;i}_{m-1}\circ \hat{\pd}^{{n-1};j-1}_m$, for $1\leq i\leq n-2$ and $i+2\leq j\leq n$. We have
\begin{gather*}
\hat{\pd}^{n;j}_{m-1}\bigl(\hat{\pd}^{{n-1};i}_m\Phi\bigr)=\bigl(\Phi\circ_i E^{(2)}_V\bigr)\circ_j E^{(2)}_V,\\
\hat{\pd}^{n;i}_{m-1}\bigl(\hat{\pd}^{{n-1};j-1}_m\Phi\bigr)=\bigl(\Phi\circ_{j-1} E^{(2)}_V\bigr)\circ_i E^{(2)}_V.
\end{gather*}
 By \cite[Lemma~4.21]{Thesis}, these are equal, proving the fourth case.

5. Now we show $\hat{\pd}^{n;i+1}_{m-1}\circ \hat{\pd}^{{n-1};i}_m=\hat{\pd}^{n;i}_{m-1}\circ \hat{\pd}^{{n-1};i}_m$, for $1\leq i\leq n-1$. We have{\samepage
\begin{gather*}
\hat{\pd}^{n;i+1}_{m-1}\bigl(\hat{\pd}^{{n-1};i}_m\Phi\bigr)=\bigl(\Phi\circ_{i} E^{(2)}_V\bigr)\circ_{i+1} E^{(2)}_V,\\
\hat{\pd}^{n;i}_{m-1}\bigl(\hat{\pd}^{{n-1};i}_m\Phi\bigr)=\bigl(\Phi\circ_{i} E^{(2)}_V\bigr)\circ_{i} E^{(2)}_V.
\end{gather*}
 By \cite[Lemma 4.22]{Thesis}, these are equal, proving the fifth case.}

6. The case of $\hat{\pd}^{n;n+1}_{m-1}\circ \hat{\pd}^{{n-1};i}_m=\hat{\pd}^{n;i}_{m-1}\circ \hat{\pd}^{{n-1};n}_m$, for $1\leq i\leq n-1$, follows analogously to case (1), except using the right action.

7. Similarly, the case of $\hat{\pd}^{n;n+1}_{m-1}\circ \hat{\pd}^{{n-1};n}_m=\hat{\pd}^{n;n}_{m-1}\circ \hat{\pd}^{{n-1};n}_m$ follows similarly to case~(2), except using the right action. As in case~(2), we need associativity.

 This handles all possible cases, proving that $\hat{\pd}^{n;i}_m$ forms a family of coface maps.
\end{proof}

\subsection{Coboundary map}
We now define the coboundary map
\[\hat{\delta}^n_m\colon \ \widehat{C}^n_m(V,W)\to \widehat{C}^{n+1}_{m-1}(V,W)\]
by
\[\hat{\delta}^n_m=\sum_{i=0}^{n+1}(-1)^i\hat{\pd}^{n;i}_m.\]
We also define
\[\hat{\delta^2_{1/2}}\colon \ \widehat{C}^2_{1/2}(V,W)\to \widehat{C}^3_0(V,W)\]
by
\begin{align*}
& \bigl\langle w',\bigl(\hat{\delta}^2_{1/2}(\Phi)\bigr)(v_1\otimes v_2\otimes v_3;z_1,z_2,z_3)\bigr\rangle\\
 &\qquad =R\bigl(\bigl\langle w',E^{(1)}_W(v_1;\Phi(v_2\otimes v_3);z_1,z_2,z_3)\bigr\rangle \\
 &\qquad \quad +\bigl\langle w',\Phi\bigl(v_1\otimes E^{(2)}_V(v_2\otimes v_3;\vac);z_1,z_2,z_3\bigr)\bigr\rangle\bigr)\\
 &\qquad\quad -R\bigl(\bigl\langle w',\Phi\bigl(E^{(2)}_V(v_1\otimes v_2;\vac)\otimes v_3;z_1,z_2,z_3\bigr)\bigr\rangle \\
 &\qquad\quad +\bigl\langle w',E^{W,(1)}_{WV}(\Phi(v_1\otimes v_2);v_3;z_1,z_2,z_3)\bigr\rangle \bigr)
\end{align*}
for $w'\in W'$, $\Phi\in\widehat{C}^2_{1/2}(V,W)$, $v_1,v_2,v_3\in V$, and $(z_1,z_2,z_3)\in F_3\C$.
\begin{Remark}
The use of $\overline{W}$-valued functions is crucial to make sense of the coboundary map. As maps $V^{\otimes n}\to \widetilde{W}_{z_1, \dots , z_n}$, arithmetic operations are permitted. If we stick to the series, even with the composable condition, such arithmetic operations are not allowed since there does not exist a common region of convergence for each series appearing in the summand.
\end{Remark}

\begin{Proposition}\label{vertexHochschild}
 For $n\in\N$, $m\in \Z_++1$, with $m\geq 2$, $\hat{\delta}^{n+1}_{m-1}\circ\hat{\delta}^{n}_{m}=0$. Furthermore, $\hat{\delta}^{2}_{1/2}\circ \hat{\delta}^{1}_{2}=0$.
\end{Proposition}
\begin{proof}
 The first claim follows directly from Lemma \ref{CosimpLemma} and the corresponding discussion in \cite[Chapter~9]{Weibel}. The second follows from the first, and that $\hat{\delta}^1_2\bigl(\widehat{C}^1_2(V,W)\bigr)\sseq\widehat{C}^2_1(V,W)\sseq\widehat{C}^2_{1/2}(V,W)$.
\end{proof}

\begin{Theorem}\label{vertexHarrison}
 For $n\in\N$, $m\in\Z_+$, $\hat{\delta}^n_m(C^n_m(V,W))\sseq C^{n+1}_{m-1}(V,W)$. In addition, we have that $\hat{\delta}^2_{1/2}\bigl(C^2_{1/2}(V,W)\bigr)\sseq C^3_0(V,W)$.
\end{Theorem}
\begin{proof}
 The first claim follows from combinatorics identical to Theorem~\ref{superHarrisonthm}. As for the second claim, let $\Phi\in C^2_{1/2}(V,W)$. The shuffles of $J_1(1,2,3)$ are
 \[\begin{pmatrix}
 1 & 2 & 3\\
 1 & 2 & 3
 \end{pmatrix},\qquad \begin{pmatrix}
 1 & 2 & 3\\
 2 & 1 & 3
 \end{pmatrix},\qquad\text{and}\qquad \begin{pmatrix}
 1 & 2 & 3\\
 3 & 1 & 2
 \end{pmatrix}.\]
Also precomposing the elements of $J_2(1,2,3)$ with $\left(\begin{smallmatrix}
 1 & 2 & 3\\
 3 & 1 & 2
 \end{smallmatrix}\right)$ returns exactly the set $J_1(1,2,3)$. Therefore, it is sufficient to show that
 \[\sum_{\sigma\in J\iv_{3;1}}(-1)^\sigma(-1)^{\sigma^o}\sigma\bigl(\hat{\delta}^2_{1/2}(\Phi)\bigr)=0.\]
 Let $v_1,v_2,v_3\in V$, all $\Z_2$ homogeneous, $w'\in W'$, $(z_1,z_2,z_3)\in F_3\C$, and $\zeta\in \C$ such that $(z_i-\zeta,z_j-\zeta),(z_k-\zeta)\in F_2\C$ for $i,j,k=1,2,3$, $i<j$. Then by definition we have
 \begin{align}
 & \sum_{\sigma \in J\iv_{3,1}}(-1)^{\sigma}(-1)^{\sigma^o}\bigl\langle w',\sigma\bigl(\bigl(\hat{\delta}^2_{1/2}(\Phi)\bigr)(v_{\sigma(1)}\otimes v_{\sigma(2)}\otimes v_{\sigma(3)})\bigr)(z_1,z_2,z_3)\bigr\rangle \nonumber\\
& \qquad = \bigl\langle w',\bigl(\bigl(\hat{\delta}^2_{1/2}(\Phi)\bigr)(v_1\otimes v_2\otimes v_3)\bigr)(z_1,z_2,z_3)\bigr\rangle \nonumber\\
 & \qquad\quad -(-1)^{|v_1||v_2|}\bigl\langle w',\bigl(\bigl(\hat{\delta}^2_{1/2}(\Phi)\bigr)(v_2\otimes v_1\otimes v_3)\bigr)(z_2,z_1,z_3)\bigr\rangle \nonumber\\
 &\qquad\quad +(-1)^{|v_1|(|v_2|+|v_3|)}\bigl\langle w',\bigl(\bigl(\hat{\delta}^2_{1/2}(\Phi)\bigr)(v_2\otimes v_3\otimes v_1)\bigr)(z_2,z_3,z_1)\bigr\rangle\nonumber\\
 &\qquad = R\bigl(\bigl\langle w',\bigl(E^{(1)}_W(v_1;\Phi(v_2\otimes v_3))\bigr)(z_1,z_2,z_3)\bigr\rangle \nonumber\\
 &\qquad\quad +\bigl\langle w',\bigl(\Phi\bigl(v_1\otimes E^{(2)}_V(v_2\otimes v_3;\vac)\bigr)\bigr)(z_1,z_2,z_3)\bigr\rangle \bigr)\label{halfharrison1}\\
 &\qquad\quad -R\bigl(\bigl\langle w',\bigl(\Phi\bigl(E^{(2)}_V(v_1\otimes v_2;\vac)\otimes v_3\bigr)\bigr)(z_1,z_2,z_3)\bigr\rangle \nonumber\\
 & \qquad\quad +\bigl\langle w',\bigl(E^{W,(1)}_{WV}(\Phi(v_1\otimes v_2);v_3)\bigr)(z_1,z_2,z_3)\bigr\rangle \bigr)\label{halfharrison2}\\
 &\qquad\quad -(-1)^{|v_1||v_2|}R\bigl(\bigl\langle w',\bigl(E^{(1)}_W(v_2;\Phi(v_1\otimes v_3))\bigr)(z_2,z_1,z_3)\bigr\rangle \nonumber\\
 & \qquad\quad +\bigl\langle w',\bigl(\Phi(v_2\otimes E^{(2)}_V(v_1\otimes v_3;\vac))\bigr)(z_2,z_1,z_3)\bigr\rangle \bigr)\label{halfharrison3}\\
 &\qquad\quad+(-1)^{|v_1||v_2|}R\bigl(\bigl\langle w',\bigl(\Phi\bigl(E^{(2)}_V(v_2\otimes v_1;\vac)\otimes v_3\bigr)\bigr)(z_2,z_1,z_3)\bigr\rangle\nonumber\\
 &\qquad\quad +\bigl\langle w',\bigl(E^{W,(1)}_{WV}(\Phi(v_2\otimes v_1);v_3)\bigr)(z_2,z_1,z_3)\bigr\rangle \bigr)\label{halfharrison4}\\
 &\qquad\quad +(-1)^{|v_1|(|v_2|+|v_3|)}R\bigl(\bigl\langle w',\bigl(E^{(1)}_W(v_2;\Phi(v_3\otimes v_1))\bigr)(z_2,z_3,z_1)\bigr\rangle \nonumber\\
 &\qquad\quad +\bigl\langle w',\bigl(\Phi\bigl(v_2\otimes E^{(2)}_V(v_3\otimes v_1;\vac)\bigr)\bigr)(z_2,z_3,z_1)\bigr\rangle \bigr)\label{halfharrison5}\\
 &\qquad\quad-(-1)^{|v_1|(|v_2|+|v_3|)}R\bigl(\bigl\langle w',\bigl(\Phi\bigl(E^{(2)}_V(v_2\otimes v_3;\vac)\otimes v_1\bigr)\bigr)(z_2,z_3,z_1)\bigr\rangle \nonumber\\
 &\qquad\quad +\bigl\langle w',\bigl(E^{W,(1)}_{WV}(\Phi(v_2\otimes v_3);v_1)\bigr)(z_2,z_3,z_1)\bigr\rangle \bigr).\label{halfharrison6}
 \end{align}
 Since $\Phi\in C^2_{1/2}(V,W)$ and $E^{(2)}_V\in C^2_{1/2}(V,V)$, we have
 \begin{align*}
& (\Phi(u_1\otimes u_2))(\zeta_1,\zeta_2)= (-1)^{|u_1||u_2|}(\Phi(u_2\otimes u_1))(\zeta_2,\zeta_1),\\
& \bigl(E^{(2)}_V(u_1\otimes u_2)\bigr)(\zeta_1,\zeta_2)= (-1)^{|u_1||u_2|}\bigl(E^{(2)}_V(u_2\otimes u_1)\bigr)(\zeta_2,\zeta_1)
 \end{align*}
 for $u_1,u_2\in V$, both $\Z_2$ homogeneous, and $(\zeta_1,\zeta_2)\in F_2\C$. Also,
 \[E^{(1)}_W(u;w)=(-1)^{|u||w|}E^{W,(1)}_{WV}(w;u)\]
 for $u\in V$, and $w\in W$. Using these formulas, we see $(\ref{halfharrison1})=-(\ref{halfharrison6})$, $(\ref{halfharrison2})=-(\ref{halfharrison4})$, and $(\ref{halfharrison3})=-(\ref{halfharrison5})$, so the whole sum goes to zero.
\end{proof}

\begin{Proposition}
 Let $f\colon V\to W$ be any homogeneous map satisfying $f \Db = \Db_W f$. Let \smash{$F\in \widehat{C}_0^1(V, W)$} be defined by
$F(v;z) = {\rm e}^{z\Db_W}f(v)$.
 Then the corresponding $\delta F\colon V\otimes V \to \widetilde{W}_{z_1z_2}$, i.e.,
 \begin{align*}
& \Phi(v_1, v_2; z_1, z_2) = (\delta F)(v_1\otimes v_2; z_1, z_2)
 = E(Y_W(v_1, z_1) F(v_2; z_2)) \\
 & \qquad{}- E(F(Y(v_1, z_1-\zeta)Y(v_2, z_2-\zeta)\one; \zeta))
 + (-1)^{|v_2||f(v_1)|}E(Y_W(v_2, z_2)F(v_1;z_1)).
 \end{align*}
 is $\frac 1 2$-composable, i.e., $\delta F \in \widehat{C}^2_{1/2}(V, W)$.
\end{Proposition}

\begin{proof}Instead of elaborating a technical computation, we explain the conceptual aspect here. It is clear that \smash{$F \in \widehat{C}_1^1(V, W)$}. It is also clear that for $v_1, v_2, v_3\in V$, $Y_W(v_1, z_1)Y_W(v_2, z_2) f(v_3)$ and $f(Y(v_1, z_1)Y(v_2, z_2)v_3)$ all converge absolutely when $|z_1|>|z_2|>0$ to rational functions with the only possible poles at $z_1=0$, $z_2=0$ and $z_1=z_2$. The problematic series is $Y_W(v_1, z_1)f(Y(v_2, z_2)v_3)$. It is very easy to adjust~$f$ so that this series converges nowhere. This series (with ${\rm e}^{\zeta D}$-modifications) does appear in the individual summands
\begin{gather*}
\sum_{r\in \C}E_W^{(1)}(v_1; P_r(\delta F)(v_2\otimes v_3; z_2-\zeta, z_3-\zeta); z_1, \zeta),\\
\sum_{r\in \C}(\delta F)(v_1\otimes P_r E(Y(v_2, z_2-\zeta)Y(v_3, z_3-\zeta)\one); z_1, \zeta).
\end{gather*}
 So we are not able to conclude that each individual summand converges. However, if we add up these two summands, the problematic terms cancel out, thus do not appear in the sum. The same applies to the other two summands.
\end{proof}

Note that for $m\in\Z_+$, we have \smash{$\widehat{C}^n_\infty(V,W)\sseq \widehat{C}^n_m(V,W)$} for all $n\in\N$, so we can define a map
\[\hat{\delta}^n_\infty={\hat{\delta}^n_m\bigr|}_{\widehat{C}^n_\infty(V,W)}\colon \ \widehat{C}^n_\infty(V,W)\to \widehat{C}^{n+1}_\infty(V,W),\]
which is independent of $m$. Now, for $n\in\N$, $m\in\Z_+$ or $m=\frac{1}{2}$ or $m=\infty$, we define $\delta^n_m=\smash{{\hat{\delta}^n_m\bigr|}_{C^n_m(V,W)}}$. By Proposition~\ref{vertexHochschild} and Theorem~\ref{vertexHarrison}, for $m\in\Z_+$ we obtain chain complexes:
\[\begin{tikzcd}
	0 & {C^0_m(V,W)} & {C^1_{m-1}(V,W)} & \cdots & {C^m_0(V,W)} & 0, \\
	0 & {C^0_3(V,W)} & {C^1_2(V,W)} & {C^2_{1/2}(V,W)} & {C^3_0(V,W)} & 0, \\
	0 & {C^0_\infty(V,W)} & {C^1_\infty(V,W)} & {C^2_\infty(V,W)} & {C^3_\infty(V,W)} & \cdots.
	\arrow[from=1-1, to=1-2]
	\arrow["{\delta^0_m}", from=1-2, to=1-3]
	\arrow["{\delta^1_{m-1}}", from=1-3, to=1-4]
	\arrow["{\delta^{m-1}_1}", from=1-4, to=1-5]
	\arrow[from=1-5, to=1-6]
	\arrow[from=2-1, to=2-2]
	\arrow["{\delta^0_3}", from=2-2, to=2-3]
	\arrow["{\delta^1_2}", from=2-3, to=2-4]
	\arrow["{\delta^2_{1/2}}", from=2-4, to=2-5]
	\arrow[from=2-5, to=2-6]
	\arrow[from=3-1, to=3-2]
	\arrow["{\delta^0_\infty}", from=3-2, to=3-3]
	\arrow["{\delta^1_\infty}", from=3-3, to=3-4]
	\arrow["{\delta^2_\infty}", from=3-4, to=3-5]
	\arrow["{\delta^3_\infty}", from=3-5, to=3-6]
\end{tikzcd}\]
\begin{Definition}
 Let $n\in\N$. For $m\in\Z_+$, we define the {\it $n$-th cohomology $H^n_m(V,W)$ of $V$ with coefficients in $W$ and composable with $m$ vertex operators} to be
 \[H^n_m(V,W)=\ker{\delta^n_m}/\operatorname{im}{\delta^{n-1}_{m+1}}.\]
 We also define
 \[H^2_{1/2}(V,W)=\ker{\delta^2_{1/2}}/\operatorname{im}{\delta^1_2}\qquad
 \text{and}\qquad
H^n_\infty(V,W)=\ker{\delta^n_\infty}/\operatorname{im}{\delta^{n-1}_\infty}.\]
\end{Definition}

\begin{Remark}
 As in Remark~\ref{HarrisonRestriction}, when restricted to vertex algebras (i.e., $V^{1}=0$) and even modules, this cohomology theory agrees with that of Huang in~\cite{H-Coh}.
\end{Remark}

\section{First and second cohomologies and first-order deformations}\label{section4}

In this section, we summarize the results in~\cite{H-1st-2nd-Coh}, which through minor modifications remain true in the context of grading-restricted vertex superalgebras. More explicitly, the proofs in \cite{H-1st-2nd-Coh,H-1st-2nd-Coh-Add} involve commuting at most two elements (say $v$ and $w$), so when dealing with $\Z_2$-homogeneous elements, terms are modified by $(-1)^{|v||w|}$.

\subsection{First cohomology and derivations} Let $V$ be a grading-restricted vertex algebra and $W$ a $V$-supermodule.

\begin{Definition}
A $\bigl(\frac{1}{2}\Z\times \Z_2\bigr)$-grading-preserving linear map $f\colon V\to W$ is called a {\it derivation} if, for $u,v\in V$, both $\Z_2$-homogeneous,
\begin{align*}
 f(Y(u,x)v)&=Y^W_{WV}(f(u),x)v+Y_W(u,x)f(v)\\
 &=(-1)^{|u||v|}{\rm e}^{x\Db_W}Y_W(v,-x)f(u)+Y_W(u,x)f(v).
\end{align*}
We denote the space of all such derivations by $\Der{(V,W)}$.
\end{Definition}

We have the following result for the first cohomologies of $V$ with coefficients in $W$.
\begin{Theorem}
 $H^1_m(V,W)$ is linearly isomorphic to the space of derivations from $V$ to $W$ for any $m\in\Z_+\cup \{\infty\}$. In particular, $H^1_m(V,W)$ for $m\in\Z_+\cup\{\infty\}$ are isomorphic, and can be denoted by the same notation $H^1(V,W)$. \label{first-coh}
\end{Theorem}

\begin{proof}
First note that a derivation $f\colon V\to W$ satisfies $f(\vac)=0$. The proof follows identically as in~\cite{H-1st-2nd-Coh}. We set up the bijection between derivations and 1-cocycles. Let $m\in \Z_+$. Given a function $\Phi\in Z^1_m(V,W)$, we obtain a map $(\Phi(\cdot))(0)\colon V\to W$. On the other hand, let $f\in\Der{(V,W)}$. Then we obtain a map $\Phi_f\colon V\to \widetilde{W}_{z_1}$ given by
\begin{align*}
 (\Phi_f(v))(z)&=f(Y(v,z)\vac)=Y^W_{WV}(f(v),z)\vac + Y_W(v,z)f(\vac)\\
 &=Y^W_{WV}(f(v),z)\vac = {\rm e}^{zD_W}f(v),
\end{align*}
where the last equality follows from $|\vac|=0$.

First we show $(\Phi(\cdot))(0)\in\Der{(V,W)}$. Note, for $v\in V_{(n)}$ and $z\in\C^\times$, by the $\db$-conjugation property,
\begin{align*}
 z^{\db_W}(\Phi(v))(0) =\bigl(\Phi\bigl(z^{\db_V}v\bigr)\bigr)(z0)
 =z^n(\Phi(v))(0),
\end{align*}
so $(\Phi(v))(0)\in W_{(n)}$, meaning $(\Phi(\cdot))(0)$ preserves the $\frac{1}{2}\Z$-grading. Together with the assumption that $\Phi$ is parity-preserving, $(\Phi(\cdot))(0)$ preserves $\bigl(\frac{1}{2}\Z \times \Z_2\bigr)$-grading.

Now let $v_1,v_2\in V$ and $w'\in W'$. Since $\delta^1_m\Phi=0$,
\begin{align*}
& 0 = R(\lrgl{w',Y_W(v_1,z_1)(\Phi(v_2))(z_2)}) - R(\lrgl{w',(\Phi(Y(v_1,z_1-z_2)v_2))(z_2)})\\
 &\hphantom{0 =}{} +(-1)^{|v_1||v_2|}R(\lrgl{w',Y_W(v_2,z_2)(\Phi(v_1))(z_1)}).
\end{align*}
By the $\Db$-derivative property, we have
\begin{align*}
 R(\lrgl{w',Y_W(v_2,z_2)(\Phi(v_1))(z_1)})=R(\lrgl{w',{\rm e}^{z_1\Db_W}Y_W(v_2,-z_1+z_2)(\Phi(v_1))(0)}).
\end{align*}
Therefore, by setting $z_2=0$ we have
\begin{align*}
 &0 = R(\lrgl{w',Y_W(v_1,z_1)(\Phi(v_2))(0)}) - R(\lrgl{w',(\Phi(Y(v_1,z_1)v_2))(0)})\\
 &\hphantom{0 =}{} +(-1)^{|v_1||v_2|}R\bigl(\bigl\langle w',{\rm e}^{z_1\Db_W}Y_W(v_2,-z_1)(\Phi(v_1))(0)\bigr\rangle \bigr).
\end{align*}
Because $w'$ is arbitrary, using skew-supersymmetry,
\begin{align*}
 (\Phi(Y(v_1,z_1)v_2))(0)=Y_W(v_1,z_1)(\Phi(v_2))(0)+Y^W_{WV}((\Phi(v_1))(0),z_1)v_2,
\end{align*}
so $(\Phi(\cdot))(0)$ is a derivation from $V$ to $W$.

Now we show that $\Phi_f$ is a 1-cocycle. We first show that the map $v\mapsto Y^W_{WV}(f(v),z_1)\vac = {\rm e}^{zD_W}f(v)$ is composable with any number of vertex operators, i.e., $\Phi_f\in C^1_m(V,W)$ for all $m\in\Z_+$. Clearly, for every $m\in \Z_+$, $v_1, \dots , v_m \in V$, the series
\begin{gather*}
Y_W(v_1, z_1)\cdots Y_W(v_{m-1},z_{m-1}){\rm e}^{z_m \Db_W}f(v_m),\\
 {\rm e}^{\zeta \Db_W}f(Y(v_1, z_1-\zeta)\cdots Y(v_{m-1}, z_{m-1}-\zeta)Y(v_m,z_m-\zeta)\vac)
\end{gather*}
converges, since the product of vertex operator converges, and that ${\rm e}^{z \Db_W}$ has no impact on the convergence (see \cite{Q-Mod}). To show the convergence of the series
\[Y_W(v_1, z_1) \cdots Y_W(v_k, z_k) {\rm e}^{\zeta \Db_W} f(Y(v_{k+1}, z_{k+1}-\zeta)\cdots Y(v_{m}, z_m-\zeta)\vac), \]
we need to use the fact that $f$ is a derivation, to split it as
\begin{align}
& Y_W(v_1, z_1)\cdots Y_W(v_k, z_k) {\rm e}^{\zeta \Db_W} Y_W(v_{k+1}, z_{k+1}-\zeta)\nonumber\\
& \qquad \quad \times f(Y(v_{k+2}, z_{k+2}-\zeta)\cdots Y(v_m, z_m-\zeta)\vac) \label{Der-conv-1}\\
&\qquad + Y_W(v_1, z_1)\cdots Y_W(v_k, z_k) {\rm e}^{\zeta \Db_W} Y_{WV}^W(f(v_{k+1}), z_{k+1}-\zeta)\nonumber\\
&\qquad \quad \times Y(v_{k+2}, z_{k+2}-\zeta) \cdots Y(v_m, z_m-\zeta)\vac. \label{Der-conv-2}
\end{align}
The convergence of (\ref{Der-conv-1}) follows from an inductive argument. The convergence of (\ref{Der-conv-2}) follows from the convergence of products of vertex operators.

Now we calculate $\delta_m^1 \Phi_f$. For $v_1,v_2\in V$ and $w'\in W'$,
\begin{align}
& \bigl(\bigl(\delta^1_m\Phi_f\bigr)(v_1\otimes v_2))(z_1,z_2) \label{first-coh-term0}\\
&\qquad =R\bigl(\bigl\langle w',Y_W(v_1,z_1)Y^W_{WV}(f(v_2),z_2)\vac\bigr\rangle \bigr) \label{first-coh-term1} \\
 &\qquad\quad - R\bigl(w',Y^W_{WV}(f(Y(v_1,z_1-z_2)v_2),z_2)\vac\bigr) \label{first-coh-term2} \\
 &\qquad\quad+ (-1)^{|v_1||v_2|}R\bigl(\bigl\langle w',Y_W(v_2,z_2)Y^W_{WV}(f(v_1),z_1)\vac\bigr\rangle \bigr). \nonumber
\end{align}
With a similar calculation as in \cite{H-1st-2nd-Coh}, (\ref{first-coh-term2}) is equal to
\begin{align}
 -R\bigl(\bigl\langle w',Y^W_{WV}(f(v_1),z_1)Y_W(v_2,z_2)\vac\bigr\rangle\bigr) 
 -R\bigl(\bigl\langle w',Y_W(v_1,z_1)Y^W_{WV}(f(v_2),z_2)\vac\bigr\rangle \bigr). \label{first-coh-term4}
\end{align}
We see that terms (\ref{first-coh-term1}) and (\ref{first-coh-term4}) cancel. What remains in (\ref{first-coh-term0}) is
\begin{align*}
 - R\bigl(\bigl\langle w',Y^W_{WV}(f(v_1),z_1)Y(v_2,z_2)\vac\bigr\rangle \bigr)
 + (-1)^{|v_1||v_2|}R\bigl(\bigl\langle w',Y_W(v_2,z_2)Y^W_{WV}(f(v_1),z_1)\vac\bigr\rangle \bigr).
\end{align*}
By duality, these two terms equal each other, so the sum is zero. Thus $\Phi_f$ is a 1-cocycle. Note also that there exists no nonzero 1-coboundaries, since for every $w\in W = C_{m+1}^0(V, W)$,
$\bigl(\delta^0_{m+1} w\bigr)(v;z) = 0$. Therefore, $\Phi_f\in H^1_m(V,W)$. The maps between $\Der{(V,W)}$ and $H^1_m(V,W)$ are clearly inverse to each other. The conclusion thus follows for any $m\in\Z_+$. Since $H^1_\infty(V,W)$ is the inverse limit of the system \smash{$\bigl\{H^1_m(V,W)\bigr\}_{m\in{\Z_+}}$} and each $H^1_m(V,W)$ is isomorphic to $\Der(V,W)$, it follows that $H^1_\infty(V,W)\cong\Der(V,W)$.
\end{proof}

\subsection{Square-zero extensions}

\begin{Definition}
 Let $V$ be a grading-restricted vertex superalgebra. A {\it square-zero ideal of $V$} is an ideal $W\sseq V$ such that for any $u,v\in W$, $Y(u,x)v=0$.
\end{Definition}
\begin{Definition}
 Let $V$ be a grading-restricted vertex superalgebra and $W$ be a $V$-supermodule whose $\db_W$-grading is given by $\frac{1}{2}\Z$. A {\it square-zero extension $(\Lambda,f,g)$ of $V$ by $W$} is a grading-restricted vertex superalgebra $\Lambda$ together with a surjective homomorphism $f\colon \Lambda\to V$ of grading-restricted vertex superalgebras such that $\ker{f}$ is a square-zero ideal of $\Lambda$ and an injective homomorphism of $V$-supermodules $g:W\to\Lambda$ such that $g(W)=\ker{f}$ and that $\Db_\Lambda\circ g = g\circ \Db_W$ and $f\circ \Db_\Lambda = \Db_V \circ f$. We say two square-zero extensions $(\Lambda_1,f_1,g_1)$ and $(\Lambda_2,f_2,g_2)$ are {\it equivalent} if there exists an isomorphism of grading-restricted vertex superalgebras $h\colon \Lambda_1\to\Lambda_2$ such that the following diagram commutes:

\begin{equation} \label{sq-z-diagram}
\begin{tikzcd}
	0 & W & {\Lambda_1} & V & 0 \\
	0 & W & {\Lambda_2} & V & 0.
	\arrow[from=1-1, to=1-2]
	\arrow["{g_1}"', from=1-2, to=1-3]
	\arrow[equals, from=1-2, to=2-2]
	\arrow["{f_1}"', from=1-3, to=1-4]
	\arrow["h"', from=1-3, to=2-3]
	\arrow[from=1-4, to=1-5]
	\arrow[equals, from=1-4, to=2-4]
	\arrow[from=2-1, to=2-2]
	\arrow["{g_2}"', from=2-2, to=2-3]
	\arrow["{f_2}"', from=2-3, to=2-4]
	\arrow[from=2-4, to=2-5]
\end{tikzcd}
\end{equation}
\end{Definition}

\begin{Theorem}
 Let $V$ be a grading-restricted vertex superalgebra and $W$ a $\bigl(\frac{1}{2}\Z\times \Z_2\bigr)$-graded $V$-supermodule. Then the set of equivalence classes of square-zero extensions of $V$ by $W$ corresponds bijectively to~$H^2_{1/2}(V,W)$. \label{second-coh}
\end{Theorem}

\begin{proof}
Per the discussion in \cite{H-1st-2nd-Coh}, any square-zero extension of $V$ by $W$ is equivalent to one of the form
\[\begin{tikzcd}
	0 & W & {V\oplus W} & V & 0
	\arrow[from=1-1, to=1-2]
	\arrow["{i_2}", from=1-2, to=1-3]
	\arrow["{p_1}", from=1-3, to=1-4]
	\arrow[from=1-4, to=1-5]
\end{tikzcd}\]
for embedding $i_2\colon w\mapsto (0, w)$ and projection $p_1\colon (v, w)\mapsto v$. For such a representative of the equivalence classes of square-zero extensions, we denote the vertex operator by $Y_{V\oplus W}$. Then for $u,v\in V$, there exists an element $\Psi(u,x)v\in W(\!(x)\!)$ such that for $v_1,v_2\in V$ and $w_1,w_2\in W$,
\begin{align}
 Y_{V\oplus W} ((v_1,w_1),x)(v_2,w_2)
 =(Y(v_1,x)v_2,Y_W(v_1,x)w_2+Y^W_{WV}(w_1,x)v_2+\Psi(v_1,x)v_2). \label{sq-zero-product}
\end{align}
The vacuum of $V\oplus W$ is $(\vac,0)$, and a quick check shows $\Psi(v,x)\vac=0$ for $v\in V$.

We proceed to show that $\Psi$ leads to a 2-cocycle. By the $\Db$-derivative and $\db$-commutator properties of $V\oplus W$, we obtain (shown explicitly in \cite{H-1st-2nd-Coh})
\begin{align*}
& \frac{{\rm d}}{{\rm d}x}\Psi(v,x)=\Psi(\Db_Vv,x)=\Db_W\Psi(v,x)-\Psi(v,x)\Db_V,\\
& \db_W\Psi(v,x)-\Psi(v,x)\db_V=\Psi(\db_Vv,x)+x\frac{{\rm d}}{{\rm d}x}\Psi(v,x),\\
& z^{\db_W}\Psi(v,x)=\Psi(z^{\db_V}v,zx)z^{\db_V}.
\end{align*}
Identify $(V\oplus W)'$ with $V'\oplus W'$. For $v_1,v_2\in V$ and $w'\in W'$,
\begin{align*}
& \lrgl{(0,w'),Y_{V\oplus W}((v_1,0),z_1)Y_{V\oplus W}((v_2,0),z_2)(\vac,0)} \\
 &\qquad{}=\lrgl{w',\Psi(v_1,z_1)Y(v_2,z_2)\vac + Y_W(v_1,z_1)\Psi(v_2,z_2)\vac}
 =\lrgl{w',\Psi(v_1,z_1)Y(v_2,z_2)\vac}, \\
& (-1)^{|(v_1,0)||(v_2,0)|}\lrgl{(0,w'),Y_{V\oplus W}((v_2,0),z_2)Y_{V\oplus W}((v_1,0),z_1)(\vac,0)} \\
 &\qquad =(-1)^{|v_1||v_2|}\lrgl{w',\Psi(v_2,z_2)Y(v_1,z_1)\vac + Y_W(v_2,z_2)\Psi(v_1,z_1)\vac}\\
 &\qquad =(-1)^{|v_1||v_2|}\lrgl{w',\Psi(v_2,z_2)Y(v_1,z_1)\vac}, \\
& \lrgl{(0,w'),Y_{V\oplus W}(Y_{V\oplus W}((v_1,0),z_1-z_2)(v_2,0),z_2)(\vac,0)} \\
 &\qquad =\lrgl{w',Y^W_{WV}(\Psi(v_1,z_1-z_2)v_2,z_2)\vac + \Psi(Y(v_1,z_1-z_2)v_2,z_2)\vac}\\
 &\qquad =\lrgl{w',Y^W_{WV}(\Psi(v_1,z_1-z_2)v_2,z_2)\vac}
\end{align*}
are absolutely convergent in the regions $|z_1|>|z_2|>0$, $|z_2|>|z_1|>0$, $|z_2|>|z_1-z_2|>0$, respectively, to a common rational function in $z_1$ and $z_2$ with the only possible poles at $z_1=0$, $z_2=0$ and $z_1=z_2$. In terms of the $E$-notation, we have
\begin{align*}
\begin{split}
E(\Psi(v_1,z_1)Y(v_2,z_2)\vac)
 & = (-1)^{|v_1||v_2|}E(\Psi(v_2,z_2)Y(v_1,z_1)\vac) \\
 & = E(Y^W_{WV}(\Psi(v_1,z_1-z_2)v_2,z_2)\vac)
\end{split}
\end{align*}
in $\widetilde{W}_{z_1,z_2}$. For $v_1,v_2\in V$ and $(z_1,z_2)\in F_2\C$, we define the linear map
$\Phi\colon V\otimes V\to \widetilde{W}_{z_1,z_2}$
by
\[(\Phi(v_1\otimes v_2))(z_1,z_2)=E(\Psi(v_1,z_1)Y(v_2,z_2)\vac).\]
Straightforward arguments in \cite{H-1st-2nd-Coh} show that $\Phi$ has both the $\db$-commutator and $\Db$-derivative properties, and the weakened $\frac{1}{2}$-composability condition, meaning $\Phi\in\widehat{C}^2_{1/2}(V,W)$. The argument in \cite{H-1st-2nd-Coh} also shows that \smash{$\delta^2_{1/2}(\Phi)=0$}, as it only relies on the associativity of the vertex operators. To see that \smash{$\Phi\in C_{1/2}^2(V, W)$}, we note that for $v_1,v_2\in V$ and $(z_1,z_2)\in F_2\C$ we have
\begin{align*}
 (\Phi(v_1\otimes v_2))(z_1,z_2)&=E(\Psi(v_1,z_1)Y(v_2,z_2)\vac)
 =(-1)^{|v_1||v_2|}E(\Psi(v_2,z_2)Y(v_1,z_1)\vac)\\
 &=(-1)^{|v_1||v_2|}(\Phi(v_2\otimes v_1))(z_2,z_1)
 =(-1)^{(\sigma)^o}(\sigma(\Phi(v_2\otimes v_1)))(z_1,z_2),
\end{align*}
where $\sigma = (12)$ is the only nontrivial shuffle on 2 elements. Thus
\[\sum_{\sigma\in J\iv_{n;p}}(-1)^\sigma(-1)^{\sigma^o}\sigma(\Phi(v_{1}\otimes v_{2}))=0.\]
Hence $\Phi$ is a 2-cocycle in \smash{$C_{1/2}^2(V, W)$}.
Conversely, let $\Phi$ be a representative of an element of $H^2_{1/2}(V,W)$ satisfying $(\Phi(v_1\otimes\vac))(z_1,z_2)=0$. The existence of such a representative is~given in~\cite{H-1st-2nd-Coh-Add} (because $\vac$ is even, the proof is verbatim). Then for any $v_1,v_2\in V$, $w'\in W'$, $\lrgl{w',(\Phi(v_1\otimes v_2))(z,0)}$ is a rational function of $z$ with the only possible pole at $z=0$, whose order is bounded above by a number independent of $w'$. This defines an element $\Psi(v_1,x)v_2\in W(\!(x)\!)$ by
\[\lrgl{w',\Psi(v_1,x)v_2}|_{x=z}=\lrgl{w',(\Phi(v_1\otimes v_2))(z,0)}\]
for $z\in\C^\times$. We define $Y_{V\oplus W}(v_1,x)v_2$ using (\ref{sq-zero-product}). Reversing the above proof, we see that $Y_{V\oplus W}$ and the vacuum $(\vac,0)$ equips $V\oplus W$ with a grading-restricted vertex superalgebra structure and with the projection $p_1\colon V\oplus W\to V$ and embedding $i_2\colon W\to V\oplus W$ is a square-zero extension of $V$ by $W$.

Now we show that two square-zero extensions are equivalent if and only if the corresponding 2-cocycles differ by a coboundary. First, let $\Phi_1,\Phi_2\in\ker{\delta^2_{1/2}}$ be two 2-cocycles obtained from square-zero extensions \smash{$\bigl(V\oplus W,Y^{(k)}_{V\oplus W},p_1,i_2\bigr)$}, where $k=1,2$. Assume that $\Phi_1=\Phi_2+\delta_1(\Gamma)$, where $\Gamma\in C^1(V,W)$. Necessarily, $\Gamma$ preserves the $\db$-grading and the parity. Then we have
\begin{align*}
 &R(\lrgl{w',\Psi_1(v_1,z_1)Y(v_2,z_2)\vac}) = \lrgl{w',(\Phi_1(v_1\otimes v_2))(z_1,z_2)}\\
 &\qquad=\lrgl{w',(\Phi_2(v_1\otimes v_2))(z_1,z_2)}+\lrgl{w',(\delta_1(\Gamma))(z_1,z_2)}\\
 &\qquad=R(\lrgl{w',\Psi_2(v_1,z_1)Y(v_2,z_2)\vac}) + R(\lrgl{w',Y_W(v_1,z_1)(\Gamma(v_2))(z_2)})\\
 &\qquad\quad-R(\lrgl{w',(\Gamma(Y(v_1,z_1-z_2)v_2))(z_2)})\\
 &\qquad\quad+(-1)^{|v_1||v_2|}R(\lrgl{w',Y_W(v_2,z_2)(\Gamma(v_1))(z_1)})\\
 &\qquad=R(\lrgl{w',\Psi_2(v_1,z_1)Y(v_2,z_2)\vac}) + R(\lrgl{w',Y_W(v_1,z_1)(\Gamma(v_2))(z_2)})\\
 &\qquad\quad-R(\lrgl{w',(\Gamma(Y(v_1,z_1-z_2)v_2))(z_2)})\\
 &\qquad\quad+(-1)^{|v_1||v_2|}R(\lrgl{w',{\rm e}^{(z_1-z_2)\Db_W}Y_W(v_2,-z_1)(\Gamma(v_1))(z_2)}).
\end{align*}
Setting $z_2=0$, we have
\begin{align*}
& \lrgl{w',\Psi_1(v_1,z_1)v_2}=\lrgl{w',\Psi_2(v_1,z_1)v_2} + \lrgl{w',Y_W(v_1,z_1)(\Gamma(v_2))(0)}\\
 &\qquad -\lrgl{w',(\Gamma(Y(v_1,z_1)v_2))(0)}+(-1)^{|v_1||v_2|}\lrgl{w',{\rm e}^{z_1\Db_W}Y_W(v_2,-z_1)(\Gamma(v_1))(0)}.
\end{align*}
Then, using skew-symmetry, we have
\begin{align*}
\begin{split}
& \Psi_1(v_1,x)v_2= \Psi_2(v_1,x)v_2+Y_W(v_1,x)(\Gamma(v_2))(0) \\
&\hphantom{\Psi_1(v_1,x)v_2=}{} -(\Gamma(Y(v_1,x)v_2))(0)+Y^W_{WV}((\Gamma(v_1))(0),x)v_2.
\end{split}
\end{align*}
Using this and (\ref{sq-zero-product}), for $v_1,v_2\in V$ and $w_1,w_2\in W$, we have
\begin{align}
& Y^{(1)}_{V\oplus W}((v_1,w_1),x)(v_2,w_2)\nonumber\\
& \qquad =(Y(v_1,x)v_2,Y_W(v_1,x)w_2+Y^W_{WV}(w_1,x)v_2+\Psi_1(v_1,x)v_2)\nonumber\\
&\qquad =(Y(v_1,x)v_2,Y_W(v_1,x)w_2+Y^W_{WV}(w_1,x)v_2+\Psi_2(v_1,x)v_2\nonumber\\
 &\qquad\quad{} +Y_W(v_1,x)(\Gamma(v_2))(0)-(\Gamma(Y(v_1,x)v_2))(0)+Y^W_{WV}((\Gamma(v_1))(0),x)v_2)\nonumber\\
&\qquad{} =Y^{(2)}_{V\oplus W}((v_1,w_1+(\Gamma(v_1))(0)),x)(v_2,w_2+(\Gamma(v_2))(0))
 -(0,(\Gamma(Y(v_1,x)v_2))(0)).\!\!\! \label{sq-zero-coboundary}
\end{align}

We now define a linear map $h\in\End{(V\oplus W)}$ by, for $v\in V$ and $w\in W$
\[h(v,w)=(v,w+(\Gamma(v))(0)).\]
Then $h$ is a linear isomorphism, and (\ref{sq-zero-coboundary}) can be rewritten as
\[h(Y^{(1)}_{V\oplus W}((v_1,w_1),x)(v_2,w_2))=Y^{(2)}_{V\oplus W}(h(v_1,w_1),x)h(v_2,w_2).\]
Thus $h$ is an isomorphism of grading-restricted vertex superalgebras from $\bigl(V\oplus W,Y^{(1)}_{V\oplus W},(\vac,0)\bigr)$ to $\bigl(V\oplus W,Y^{(2)}_{V\oplus W},(\vac,0)\bigr)$, such that~(\ref{sq-z-diagram}) is commutative (with $\Lambda_1=V\oplus W=\Lambda_2$). Then the two square-zero extensions of~$V$ by~$W$ are equivalent.

Conversely, let $\bigl(V\oplus W,Y^{(k)}_{V\oplus W},p_1,i_2\bigr)$, where $k=1,2$, be two equivalent square-zero extensions of $V$ by $W$. So there exists an isomorphism $h\colon V\oplus W\to V\oplus W$ of grading-restricted vertex superalgebras such that~(\ref{sq-z-diagram}) commutes. It is proved in \cite[Lemma~3.1]{H-1st-2nd-Coh} that there exists a linear map $g\colon V\to W$ such that
$h(v,w)=(v,w+g(v))$
 for $v\in V$ and $w\in W$. Let $(\Gamma(v))(z_1)={\rm e}^{z_1\Db_W}g(v)\in\overline{W}$. Then $\Gamma\colon V\to\widetilde{W}_{z_1}$ is an element of $C^1_2(V,W)$. By definition, we have $(\Gamma(v))(0)=g(v)$, and $h(v,w)=(v,w+(\Gamma(v))(0))$. Let $\Phi_1$ and $\Phi_2$ be the 2-cocycles obtained from $\bigl(V\oplus W,Y^{(1)}_{V\oplus W},p_1,i_2\bigr)$ and $\bigl(V\oplus W,Y^{(2)}_{V\oplus W},p_1,i_2\bigr)$, respectively. Then following the above proof in reverse, we see that $\Phi_1=\Phi_2+\delta^1_2\Gamma$, meaning that $\Phi_1$ and $\Phi_2$ are in fact representatives of the same cohomology class.
\end{proof}

\subsection{First-order deformations}
In the case of associative algebras, square-zero extensions of an algebra by itself are useful in deformation theory, as their equivalence classes are in bijection with the equivalence classes of first-order deformations. It turns out the same results follow analogously for grading-restricted vertex (super)algebras.{\samepage
\begin{Definition}
A first-order deformation of a grading-restricted vertex superalgebra $(V, Y, \vac)$ is defined by an operator
$Y_1\colon V\otimes V \to V((x))$,
such that the $\bigl(\frac{1}{2}\Z\times \Z_2\bigr)$-graded vector space $V^t = \C[t]/\bigl(t^2\bigr)\otimes_\C V$, with the vertex operator
\[Y^t(u, x) v = Y(u, x)v + t Y_1(u, x)v,\]
the $\Db$ operator $\Db_V$, and the vacuum element $\vac$, i.e., $\bigl(V^t, Y^t, \vac\bigr)$, forms a grading-restricted vertex superalgebra over the base ring $\C[t]/\bigl(t^2\bigr)$ such that $\Db_{V^t}=\Db_V$. Two first-order deformations $\bigl(V^t, Y^{t,(1)}\bigr)$ and $\bigl(V^t, Y^{t,(2)}\bigr)$ are \textit{equivalent}, if there exists a $\C[t]/\bigl(t^2\bigr)$-linear isomorphism $f^t\colon V \oplus tV = V^t \to V^t$ of vertex superalgebras, whose restriction on~$V$ is of the form
$f^t|_V = 1_V + t f_1$,
where $f_1\colon V\to V$ is a $\C$-linear grading-preserving map. In other words, for $u,v\in V$,
\[f^t(u+tv) = u + t(v+f_1(u)). \]
To emphasize, equivalent first-order deformations are isomorphic, but the converse does not necessarily hold.
\end{Definition}

\begin{Remark}When working with first-order deformations, we suppress the notation that we are working modulo~$t^2$, for convenience.
\end{Remark}

}

\begin{Theorem}\label{sq-z-ext}
 The set of equivalence classes of first-order deformations of a grading-restricted vertex superalgebra is in bijection with the set of equivalence classes of square-zero extensions of~$V$ by~$V$.
\end{Theorem}

Theorems~\ref{second-coh} and~\ref{sq-z-ext} then imply the following result.
\begin{Theorem}\label{deformation-result}
 Let $V$ be a grading-restricted vertex superalgebra. Then the set of equivalence classes of first-order deformations of~$V$ correspond bijectively to $H^2_{1/2}(V,V)$.
\end{Theorem}

Theorem~\ref{deformation-result} is a significant result of this paper; it allows for a method of computing first-order deformations of grading-restricted vertex superalgebras by computing the cohomology group $H_{1/2}^2(V, V)$. Now we move to proving Theorem~\ref{sq-z-ext}, following alongside \cite{H-1st-2nd-Coh,H-1st-2nd-Coh-Add}.

\begin{proof}Let $Y_t\colon V\otimes V\to V(\!(x)\!)$, $t\in\C$ be a first-order deformation of $V$. By definition, there exists
$\Psi\colon V\otimes V\to V(\!(x)\!)$, $v_1\otimes v_2 \to \Psi(v_1,x)v_2$
such that for $v_1,v_2\in V$,
\[Y_t(v_1,x)v_2=Y(v_1,x)v_2+t\Psi(v_1,x)v_2\]
and $(V,Y_t,\vac)$ is a family of grading restricted vertex algebras up to the first order in $t$. By the identity property and creation properties for $(V,Y_t,\vac)$ up to the first order in $t$, we can obtain
\begin{gather}
 \Psi(\vac,x)v=0, \label{fod-first-vac}
\\
 \lim_{x\to 0}\Psi(v,x)\vac=0 \label{fod-second-vac}
\end{gather}
for $v\in V$, respectively.

The duality property up to the first order in $t$ can be written explicitly as follows: For $v_1,v_2,v_3\in V$ and $v'\in V'$,
\begin{align*}
& \lrgl{v',(Y(v_1,z_1)\Psi(v_2,z_2)+\Psi(v_1,z_1)Y(v_2,z_2))v_3},\\
& (-1)^{|v_1||v_2|}\lrgl{v',(Y(v_2,z_2)\Psi(v_1,z_1)+\Psi(v_2,z_2)Y(v_1,z_1))v_3},\\
& \lrgl{v',(Y(\Psi(v_1,z_1-z_2)v_2,z_2)+\Psi(Y(v_1,z_1-z_2)v_2,z_2))v_3}
\end{align*}
are absolutely convergent in the regions $|z_1|>|z_2|>0$, $|z_2|>|z_1|>0$, and $|z_2|>|z_1-z_2|>0$, respectively, to a common rational function in $z_1$ and $z_2$ with the only possible poles at $z_1=0$, $z_2=0$ and $z_1=z_2$.

Now, using this we make a square-zero extension of $V$ by $V$. Let
\begin{align*}
& Y_{V\oplus V}\colon \ (V\oplus V)\otimes(V\oplus V)\to (V\oplus V)\bigl[\bigl[x,x\iv\bigr]\bigr],\\
& \hphantom{Y_{V\oplus V}\colon}{} \ (u_1,v_1)\otimes(u_2,v_2)\mapsto Y_{V\oplus V}((u_1,v_1),x)(u_2,v_2)
\end{align*}
be given by
\begin{align}
 Y_{V\oplus V} ((u_1,v_1),x)(u_2,v_2)
 =(Y(u_1,x)u_2,Y(u_1,x)v_2+Y(v_1,x)u_2+\Psi(u_1,x)u_2) \label{fod-prod}
\end{align}
for $u_1,u_2,v_1,v_2\in V$. In particular, by (\ref{fod-prod}), (\ref{fod-first-vac}), and (\ref{fod-second-vac}) we have
\begin{gather*}
Y_{V\oplus V}((\vac,0),x)(u,v)=(u,v),\qquad
\lim_{x\to 0}Y_{V\oplus V}((u,v),x)(\vac,0)=(u,v)
\end{gather*}
for $u,v\in V$, respectively. That is, $(V\oplus V,Y_{V\oplus V},(\vac,0))$ has the identity and creation properties. By the duality property of $V$, and duality up to the first order in $t$, both of which involve absolute convergence, we have that, for $u_1,u_2,u_3,v_1,v_2,v_3\in V$, $u',v'\in V'${\samepage
\begin{align}
& \lrgl{(u',v'),Y_{V\oplus V}((u_1,v_1),z_1)Y_{V\oplus V}((u_2,v_2),z_2)(u_3,v_3)},\nonumber\\
& \lrgl{(u',v'),Y_{V\oplus V}(Y_{V\oplus V}((u_1,v_1),z_1-z_2)(u_2,v_2),z_2)(u_3,v_3)}\label{sq-z-fod-duality}
\end{align}}%
are absolutely convergent to a common rational function in the regions $|z_1|>|z_2|>0$ and $|z_2|>|z_1-z_2|>0$ with the only possible poles at $z_1=0$, $z_2=0$, $z_1=z_2$ (\cite{H-1st-2nd-Coh} shows this explicitly).
Now, with the added assumption that $(u_1,v_1),(u_2,v_2)\in V\oplus V$ are both $\Z_2$ homogeneous elements, we have that
\begin{align*}
& (-1)^{|(u_1,v_1)||(u_2,v_2)|}\lrgl{(u',v'),Y_{V\oplus V}((u_2,v_2),z_2)Y_{V\oplus V}((u_1,v_1),z_1)(u_3,v_3)}\\
&\qquad =(-1)^{|u_1||u_2|}\lrgl{u',Y(u_2,z_2)Y(u_1,z_1)u_3}+(-1)^{|u_1||u_2|}\lrgl{v',Y(u_2,z_2)Y(u_1,z_1)v_3}\\
 &\qquad\quad +(-1)^{|v_1||u_2|}\lrgl{v',Y(u_2,z_2)Y(v_1,z_1)u_3}+(-1)^{|u_1||u_2|}\lrgl{v',Y(u_2,z_2)\Psi(u_1,z_1)u_3}\\
 &\qquad\quad +(-1)^{|u_1||v_2|}\lrgl{v',Y(v_2,z_2)Y(u_1,z_1)u_3}+(-1)^{|u_1||u_2|}\lrgl{v',\Psi(u_2,z_2)Y(u_1,z_1)u_3},
\end{align*}
which by the duality of $V$ and the duality up to the first order in $t$ is convergent to the same rational function as (\ref{sq-z-fod-duality}) in the region $|z_2|>|z_1|>0$ with the only possible poles at $z_1=0$, $z_2=0$, $z_1=z_2$. So $(V\oplus V, Y_{V\oplus V},(\vac,0))$ has the duality property.

Note that the $\Db$-derivative property is a consequence of the other axioms for vertex superalgebras (see \cite{LL}). The $\db$-commutator formula follows from the corresponding properties of $Y$ and $\Psi$. Thus $(V\oplus V,Y_{V\oplus V},(\vac,0))$ is a grading-restricted vertex superalgebra. That $(V\oplus V,Y_{V\oplus V},(\vac,0))$ is a square-zero extension of~$V$ by~$V$ follows identically as in~\cite{H-1st-2nd-Coh}.

Conversely, let $(V\oplus V,Y_{V\oplus V},p_1,i_2)$ be a square-zero extension of $V$ by $V$. Then there exists
\begin{align*}
 \Psi\colon \ V\otimes V\to V(\!(x)\!), \qquad v_1\otimes v_2\mapsto \Psi(v_1,x)v_2
\end{align*}
such that for $u_1,u_2\in V$
\[Y_{V\oplus V}((u_1,0),x)(u_2,0)=(Y(u_1,x)u_2,\Psi(u_1,x)u_2).\]
Then for $t\in \C$, define
\[Y_t(v_1,x)v_2=Y(v_1,x)v_2+t\Psi(v_1,x)v_2\]
for $v_1,v_2\in V$. Working backwards through the above proof, using the identity property, creation property, and duality of the grading-restricted superalgebra $(V\oplus V,Y_{V\oplus V},(\vac,0))$ shows that $Y_t$ is a first-order deformation of $(V,Y,\vac)$.

That two first-order deformations of $V$ are equivalent if and only if the corresponding square-zero extensions of $V$ by $V$ are equivalent follows identically as in \cite{H-1st-2nd-Coh}; one direction follows from the difference of two first-order deformations and the other from \cite[Lemma~3.1]{H-1st-2nd-Coh}.
\end{proof}

\section{An explicit example of first cohomology}\label{section5}
In this section, we use Theorem~\ref{first-coh} to compute the first cohomology of the affine vertex superalgebra $V^k(\mfk{osp}(1|2))$. The higher cohomology is however far more technical and shall be left for future studies.

\subsection[A review of osp(1|2)]{A review of $\boldsymbol{\mfk{osp}(1|2)}$}

 A \textit{Lie superalgebra} is a vector superspace $\g=\g^0\oplus\g^1$ equipped with a bilinear multiplication $[\cdot,\cdot]$ satisfying
\begin{enumerate}\itemsep=0pt
 \item \textit{Parity} for homogeneous elements $f,g\in\g$, $[f,g]\in\g^{|f|+|g|}$,
 \item \textit{Skew-supersymmetry:} For homogeneous elements $f,g\in \g$, $[f,g]=-(-1)^{|f||g|}[g,f]$,

 \item \textit{Super Jacobi identity:} For homogeneous elements $f,g,h\in\g$,
 \[[f,[g,h]]=[[f,g],h]+(-1)^{|f||g|}[g,[f,h]].\]
\end{enumerate}
Let $V=V^0\oplus V^1$ be an $m+2n$-dimensional vector superspace. Then $\text{End}(V)$ is an associative superalgebra. Let $\mathfrak{gl}(V)$ be the Lie superalgebra with elements in $\End(V)$ and bracket given by the supercommutator. The orthosymplectic Lie algebra $\mfk{osp}(m|2n)$ is the subalgebra of $\mathfrak{gl}(V)$ consisting of elements preserving a nondegenerate even supersymmetric bilinear form. For $m,n\geq1$, $\mfk{osp}(m|2n)$ is simple. From Section 1.1.3 of \cite{Cheng}, there are isomorphisms of superspaces
\begin{align*}
 \mfk{osp}(m|2n)^0 \cong\mfk{so}(V^0)\oplus\mfk{sp}(V^1),\qquad
 \mfk{osp}(m|2n)^1 \cong \C^m\otimes\C^{2n}.
\end{align*}
In particular, $\mfk{osp}(1|2)$ is a $(3+2)$-dimensional Lie superalgebra isomorphic as a superspace to $\slt\oplus\C^2$. It has a basis $\{e,f,h\}\sseq\mfk{osp}(1|2)^0$, $\{x,y\}\sseq\mfk{osp}(1|2)^1$. The nontrivial relations are
\begin{gather*}
 [h,e]=2e,\qquad[e,f]=h,\qquad[h,f]=-2f,\\
 [x,x]=2e,\qquad[x,y]=h,\qquad[y,y]=-2f\\
 [h,x]=x,\qquad[f,x]=-y,\qquad[e,y]=-x,\qquad[h,y]=-y.
\end{gather*}
The nondegenerate Killing form is rescaled so that the nonzero pairings are
\begin{align*}
\kappa(e,f)=\kappa(f,e)=1,\qquad\kappa(h,h)=2,\qquad\kappa(x,y)=-\kappa(y,x)=2.
\end{align*}
We construct the affine vertex superalgebras using the same process as that of Chapter 6 of \cite{LL}. First we take the central extension of the loop algebra
\[\widehat{\mfk{osp}}(1|2)=\mfk{osp}(1|2)\otimes_\C\C[t,t\iv]\oplus\C K.\]
with $K$ central, and for $a,b\in\mfk{osp}(1|2)$, $m,n\in\Z$,
\[[a\otimes t^m,b\otimes t^n]=[a,b]\otimes t^{m+n}+m\delta_{m+n,0}\kappa(a,b)K.\]
For $k\in\C$, let $\C_k$ be the one-dimensional module for the subalgebra spanned by $a\otimes t^n, a\in \mfk{osp}(1|2), n\geq 0$ and $K$, where $a\otimes t^n$ acts trivially, and $K$ acts by the scalar $k$. Let $V^k(\mfk{osp}(1|2))$ be the induced module of $\mathbb{C}_k$ from the subalgebra to $\widehat{\mfk{osp}}(1|2)$. Then $V^k(\mfk{osp}(1|2))$ has a canonical vertex superalgebra structure strongly generated by weight-1 elements corresponding to the basis for $\mfk{osp}(1|2)$. For an element $a\in\mfk{osp}(1|2)$ viewed as an element of $V^k(\mfk{osp}(1|2))$, we have
\[Y(a,z)=\zsum{n}a_n z^{-n-1},\]
where $a_n=a\otimes t^n$, and we abbreviate $Y(a,z)=: a(z)$.

\subsection[The first cohomology of \protect{V\^{}k(osp(1|2))}]{The first cohomology of $\boldsymbol{V^k(\mfk{osp}(1|2))}$} For convenience, we set $V=V^k(\mfk{osp}(1|2))$. We now wish to compute $H^1(V,V)$.

\begin{Theorem}For $V=V^k(\mfk{osp}(1|2))$, $H^1(V,V)$ is $3$-dimensional, and has a canonical linear isomorphism with $\slt$.
\end{Theorem}

\begin{proof}From Theorem~\ref{first-coh}, it suffices to compute $\text{Der}(V, V)$. Let $F\colon V\to V$ be a derivation. Then $F$ is grading-preserving, and for $u,v\in V$, $n\in \Z$,
 \[F(u_n v) = F(u)_n v + u_n F(v).\]
In particular, for $u, v\in V_{(1)} = \mfk{osp}(1|2)$, $n= 0$, we have
 \[
 F(u_0 v) = F(u)_0 v + u_0 F(v).
 \]
Since $u_0v = [u,v]$ for $u,v\in \mfk{osp}(1|2)$, the restriction of $F$ on $V_{(1)}$ is an even derivation of $\mfk{osp}(1|2)$. It is well known that derivations of $\mfk{osp}(1|2)$ are inner, i.e., there exists $a\in \mfk{osp}(1|2)^0$ such that $F(b) = [a,b]$ (see, for example, \cite[Exercise 1.17]{Cheng}).

Since $V$ is strongly generated by $\mfk{osp}(1|2)$, $F$ is uniquely determined from its action on the generators. Moreover, its action coincides with the zero-mode of $a$, i.e., $F(b) = a_0 b$. From the super-Jacobi identity (\ref{super-Jacobi}), the zero-mode $a_0$ is a well-defined derivation on $V$. Thus, every $F\in \Der(V, V)$ coincides with the zero-mode $a_0$ of some element in $\mathfrak{sl}_2$.
\end{proof}

\subsection*{Acknowledgements}

The authors would like to thank Yi-Zhi Huang and Andrew Linshaw for the discussions on various aspects of the work. They would like to thank the anonymous referees for their thoughtful commentary and helpful suggestions. The authors would also like to express their admiration for James Lepowsky. They also acknowledge ChatGPT 5.6-Pro and 6-Pro for proofreading the manuscript and suggesting improvements, though they did not comply with all their suggestions.

\pdfbookmark[1]{References}{ref}
\LastPageEnding

\end{document}